\documentclass[notitlepage,11pt,psamsfonts]{amsart}

\usepackage{amsmath,amssymb,amsthm}
\usepackage[latin1]{inputenc}
\usepackage{verbatim}
\newtheorem{thm}{Theorem}
\newtheorem{lem}{Lemma}
\newtheorem{cor}{Corollary}
\newtheorem{prop}{Proposition}
\newtheorem{rem}{Remark}
\newtheorem{defi}{Definition}
\newcommand{\eps}{\varepsilon}
\newcommand{\R}{\mathbb{R}}

\newcommand{\N}{\mathbb{N}}
\newcommand{\Z}{\mathbb{Z}}

\renewcommand{\div}{{\rm div}\,}
\newcommand{\Id}{{\rm Id}\,}
\newcommand{\Supp}{{\rm Supp}\,}

\newcommand{\Sum}{\displaystyle \sum}

\newcommand{\Int}{\displaystyle \int}
\newcommand{\Frac}{\displaystyle \frac}

\def\ov{\overline}
\def\d{\partial}

\def\dq{\Delta_q}

\def\tilde{\widetilde}
\def\hat{\widehat}
\def\div{{\rm div}\,}
\def\s{\sigma}

\def\cC{{\mathcal C}}
\def\cF{{\mathcal F}}
\def\cH{{\mathcal H}}
\def\cL{{\mathcal L}}
\def\cP{{\mathcal P}}
\def\cQ{{\mathcal Q}}
\def\cS{{\mathcal S}}
\def\da{\delta\!a}
\def\df{\delta\!f}
\def\du{\delta\!u}
\def\dr{\delta\!\rho}
\def\dPi{\delta\!\Pi}
\begin{document}
\title[Density-dependent incompressible Euler equations]
{On the well-posedness of  the incompressible density-dependent Euler equations
in the   $L^p$ framework}

\author[R. Danchin]{Rapha\"el Danchin}
\address[R. Danchin]
{Universit\'e Paris-Est, LAMA, UMR 8050,
 61 avenue du G\'en\'eral de Gaulle,
94010 Cr\'eteil Cedex, France.}
\email{danchin@univ-paris12.fr}

\date\today
\begin{abstract}
The present paper is devoted to the study of the well-posedness
issue for the density-dependent Euler equations in the whole space. 
We establish local-in-time   results  for the Cauchy problem pertaining to data
in the Besov spaces  embedded in the set of Lipschitz functions,
 including the borderline case $B^{\frac Np+1}_{p,1}(\R^N).$ 
A continuation criterion  in the spirit of the celebrated one by  Beale-Kato-Majda in \cite{BKM}
for the classical Euler equations, is  also proved.

In contrast with the previous work dedicated to this system in the whole space, 
our approach is not restricted to the $L^2$ framework or to small perturbations
of a constant density state: we just need the density to be bounded away from zero. The key to that improvement is a new a priori estimate
in Besov spaces for an elliptic equation with nonconstant coefficients.
  \end{abstract}

\maketitle

The evolution of  the density $\rho=\rho(t,x)\in\R^+$
and of the velocity field  $u=u(t,x)\in\R^N$ of 
a nonhomogeneous incompressible fluid satisfies the following 
\emph{density-dependent Euler equations}:
\begin{equation}\label{eq:ddeuler}
\left\{\begin{array}{l}
\d_t\rho+u\cdot\nabla \rho=0,\\[1ex]
\rho(\d_tu+u\cdot\nabla u)+\nabla\Pi=\rho f,\\[1ex]
\div u=0.
\end{array}\right.
\end{equation}
Above, $f$ stands for a given body force and the gradient of the  pressure $\nabla\Pi$
is the Lagrangian multiplier associated to the divergence free constraint over the velocity. 
We assume the space variable $x$ to belong to the whole $\R^N$ with $N\geq2$. 
\smallbreak
A  plethoric number of recent mathematical works have been devoted
to the study of the classical  incompressible Euler equations
\begin{equation}\label{eq:euler}
\left\{\begin{array}{l}
\d_tu+u\cdot\nabla u+\nabla\Pi= f,\\[1ex]
\div u=0.
\end{array}\right.
\end{equation}
which may be seen   as a special case of  \eqref{eq:ddeuler}  (just take $\rho\equiv1$).
\smallbreak
In contrast, not so many works have been devoted to the study of \eqref{eq:ddeuler}
in the nonconstant density case.  In the situation  
where the  equations are considered in  a suitably smooth bounded 
 domain of $\R^2$ or $\R^3,$ the local
 well-posedness issue has been investigated by H. Beir\~ao da Veiga and A. Valli
 in \cite{BV1,BV2,BV3} for data 
 with high enough  H\"older regularity. 
The case of  data with   $W^{2,p}$ regularity  has been studied by A. Valli and  
W. Zaj\c aczkowski in \cite{VZ} and by S. Itoh and A. Tani in \cite{IT}.
 The whole space case $\R^3$ has been addressed by S. Itoh  in \cite{Itoh}.
 There,  the local existence  for initial data $(\rho_0,u_0)$ such that  $\rho_0$ is bounded, 
 bounded away from $0$ and such that $\nabla\rho_0\in H^2,$ in $u_0$ is in $H^3$ has been obtained.
 In the recent paper \cite{D4}, we have generalized \cite{Itoh}'s result to any dimension $N\geq2$
 and any Sobolev space $H^s$ with $s>1+N/2.$ Data in the limit Besov space $B^{\frac N2+1}_{2,1}$
 are also considered. 
 
 Let us also mention that, according to the work by J. Marsden in \cite{M},  
 the finite energy   solutions to~\eqref{eq:ddeuler} may be interpreted in terms of the 
action of geodesics.
This latter study is motivated by the fact that, as in the homogeneous situation, 
System \eqref{eq:ddeuler}  has a conserved energy, namely
\begin{equation}\label{eq:energy}
\|(\sqrt\rho\,u)(t)\|_{L^2}^2=\|\sqrt\rho_0\,u_0\|_{L^2}^2
+2\int_0^t\int_{\R^N} (\rho f\cdot u)(\tau,x)\,d\tau\,dx.
\end{equation}
Motivated by the fact  that, in real life, a fluid is hardly 
homogeneous,  we here want to study whether  
 the classical results for homogeneous fluids remain 
true in the nonhomogeneous framework. 
More precisely,  we aim at investigating the existence and uniqueness issue 
 in the whole space and in the $L^p$ framework
 for densities which may be large perturbations of   a constant function:  we only require 
 the density to be  bounded and bounded away from zero and to have enough regularity. 
We shall also establish blow-up criteria in the spirit of the celebrated one by
Beale-Kato-Majda criterion for \eqref{eq:euler} (see \cite{BKM}).

The functional framework that we shall adopt -- Besov spaces embedded in 
the set $C^{0,1}$ of bounded globally Lipschitz functions -- is motivated by  
the fact that  the density and velocity equations
of \eqref{eq:ddeuler} are transport equations
by the velocity field. Hence no gain of smoothness may be
expected during the evolution and conserving the initial regularity
requires the velocity field to be at least locally Lipschitz with respect
to the space variable. 
In fact, as regards the velocity field, the spaces that we shall use are exactly those
that are  suitable for \eqref{eq:euler}.  
We thus believe our results to be optimal in terms of regularity. 
\smallbreak
Now, compared to the classical Euler equations, handling the gradient of the pressure
is much more involved. To eliminate the pressure, the natural strategy consists in solving 
 the elliptic equation
 $$
 \div\bigl(a\nabla\Pi\bigr)=\div F\quad\hbox{with }\ F:=\div(f-u\cdot\nabla u)\ \hbox{ and }\
 a:=1/\rho.
 $$ 
 If $a$  is  a small perturbation of  a constant function $\ov a$ then
 the above equation may be rewritten 
 $$
 \ov a\Delta\Pi=\div\bigl((\ov a-a)\nabla\Pi\bigr)+ 
 \div F.
 $$
 Now, if $1<p<\infty$ then  the standard $L^p$ elliptic estimate 
 may be used for absorbing the first term in the right-hand side.
 Hence we expect to get the same well-posedness results as for \eqref{eq:euler}
 in this situation.  
 As a matter of fact, this strategy has been successfully implemented 
 by Y. Zhou in \cite{Zper}.

 In the general case of \emph{large} perturbations of 
 a constant density state, solving the above equation in the $\R^N$ framework
for $F\in L^p$ may be a problem (unless $p=2$ of course). 
In fact, to our knowledge, even if $a$ is smooth, bounded and bounded away from zero,
there is no solution operator $\cH:F\rightarrow\nabla\Pi$ such that 
$$
\|\nabla\Pi\|_{L^p}\leq C\|F\|_{L^p}
$$
unless $p$ is ``close'' to $2$ (see the work by N. Meyers in \cite{Meyers}). 
However that closeness is strongly related to whether $a$ itself is close to 
a constant hence no result for all $p$ may be obtained by taking advantage of
Meyers' result.

In the present work, we shall  overcome this difficulty by requiring  the data to 
satisfy a finite energy condition  so as to ensure that $F$
is in  $L^2.$ Indeed, this will enable us to  use the classical $L^2$ estimate
and from that, it turns out to be  possible to get estimates in high order Besov spaces $B^s_{p,r}.$  
\smallbreak
This observation is the conducting thread leading to the first two  well-posedness results
stated in the next section. 
The rest of the paper unfolds as follows. 
In section \ref{s:tools}, we introduce the Littlewood-Paley decomposition and recall the definition of the nonhomogeneous Besov spaces $B^s_{p,r}.$ Then,
 we define the paraproduct and remainder
operators and state a few classical results in Fourier analysis. 
Section \ref{s:elliptic} is devoted to the proof of existence results and a priori estimates 
for an elliptic equation with nonconstant coefficients in the Besov space framework. 
To our knowledge, most of the results that are presented therein are new. 
 Sections \ref{s:th:main}, \ref{s:th:2} and \ref{s:th:3} 
  are dedicated to the proof of our main existence and continuation results.
 Some technical lemmas have been postponed in the  appendix. 
 \medbreak  
\noindent\emph{Notation.}
Throughout the paper, $C$ stands for
a harmless ``constant'' whose exact meaning depends on the context.

For all  Banach space $X$ and interval $I$ of $\R,$  
we denote by $\cC(I;X)$ (resp. $\cC_b(I;X)$)
 the set of continuous  (resp. continuous bounded) functions on $I$ with
values in $X.$
If $X$ has predual $X^*$ then 
we denote by $\cC_w(I;X)$ the set of bounded measurable functions 
$f:I\rightarrow X$ such that for any $\phi\in X^*,$ the 
function $t\mapsto\langle f(t),\phi\rangle_{X\times X^*}$
is continuous over $I.$
For $p\in[1,\infty]$, the notation $L^p(I;X)$ 
stands for the set
 of measurable functions on  $I$ with values in $X$ such that
$t\mapsto \|f(t)\|_X$ belongs to $L^p(I)$.
We denote by $L^p_{loc}(I)$ the set of those functions defined on $I$ and valued in $X$
which, restricted to any compact subset $J$ of $I,$ are in $L^p(J).$

Finally, for any real valued function $a$ over $\R^N,$ we 
denote
$$
a_*:=\inf_{x\in\R^N}a(x)\quad\hbox{and}\quad
a^*:=\sup_{x\in\R^N}a(x).
$$


\section{Main results}\label{s:results}

As explained in the introduction, we hardly expect to get any well-posedness result
if the initial velocity is not in $C^{0,1}.$ 
It is well-known (see e.g. \cite{BCD},  Chap. 2) that  
the nonhomogeneous Besov space  
 $B^{s}_{p,r}$ is continuously embedded in $C^{0,1}$ 
  is and only if the  triplet $(s,p,r)\in\R\times[1,\infty]^2$  satisfies the following 
  condition: 
 $$
s>1+N/p\quad\hbox{or}\quad s\geq 1+N/p\ \hbox{ and }\ r=1.\leqno(C)
$$
This motivates the following 
 statement concerning the existence of smooth solutions with finite energy:
\begin{thm}\label{th:main}
Let $(s,p,r)$ satisfy Condition~$(C)$ with $1<p<\infty.$
Let $u_0$ be a divergence-free vector-field with coefficients
in $L^2\cap B^s_{p,r}.$ Suppose that the body force $f$ 
has coefficients in $L^1([-T_0,T_0];B^s_{p,r})\cap\cC([-T_0,T_0];L^2)$ for some $T_0>0.$
Assume that $\rho_0$ is positive, bounded and bounded away from zero and
that $\nabla\rho_0\in B^{s-1}_{p,r}.$  If $p<2,$ suppose in addition that  
$(\rho_0-\ov\rho)\in L^{p^*}$ with $p^*:=2p/(2-p)$ for some positive real number $\ov\rho.$

There exists a time $T\in(0,T_0]$ such that
 System $\eqref{eq:ddeuler}$ supplemented with initial data
$(\rho_0,u_0)$ has a unique local solution $(\rho,u,\nabla\Pi)$ on $[-T,T]\times\R^N$ with:
\begin{itemize}
\item $\rho^{\pm1}\in\cC_b([-T,T]\times\R^N),\quad
D\rho\in\cC_w([-T,T];B^{s-1}_{p,r})$ (and $(\rho-\ov\rho)\in\cC([-T,T];L^{p^*})$ if $p<2$),
\item $u\in\cC^1([-T,T];L^2)\cap\cC_w([-T,T];B^s_{p,r}),$
\item $\nabla\Pi\in\cC([-T,T];L^2)\cap L^1([-T,T];B^s_{p,r}).$
\end{itemize}
Besides, the energy equality $\eqref{eq:energy}$ is satisfied for all $t\in[-T,T],$
and time continuity holds with respect to the strong topology, if $r<\infty.$
\end{thm}
A few comments are in order:
\begin{itemize}
\item  For the classical  incompressible Euler equations \eqref{eq:euler}, 
the above result statement (without the $L^2$ assumption) belongs  to the mathematical folklore. 
It has been  established  in e.g. \cite{Z}
 in the case $1<p<\infty$ and in e.g. \cite{BCD}, Chap. 7 in the case $1\leq p\leq\infty.$
 \item Note that the above statement covers the borderline case $B^{\frac Np+1}_{p,1}$
for Condition $(C)$ \emph{without any smallness assumption}. Thus, up to 
the lower order $L^2$ assumption which is needed to control
the low frequencies of the pressure,  it extends the
result \cite{Zper} by Y. Zhou mentioned in the introduction.
 \item If one makes the stronger assumption that
  $(\rho_0-\ov\rho)\in B^s_{p,r}$  for some
 positive constant $\ov\rho$ then we get in addition 
 $(\rho-\ov\rho)\in\cC([-T,T];B^s_{p,r})$ (or 
 $\cC_w([-T,T];B^s_{p,r})$ if $r=\infty$).
 \item If $1<p\leq2$ then  $u_0\in B^s_{p,r}$ implies that $u_0\in L^2.$
 Furthermore, in dimension $N\geq3,$ the assumption that $(\rho_0-\ov\rho)\in L^{p^*}$
 may be omitted if $p>N/(N-1).$
 Therefore, except if $N=2$ and $p<2$ or if $N\geq3$ and $p\leq N/(N-1),$ the density
 need not to tend to some constant at infinity. 
\item In contrast with the homogeneous case, in dimension $N=2,$  the  global well-posedness
issue for \eqref{eq:ddeuler} with nonconstant density 
is an open (and challenging) problem.
 Indeed,  the vorticity $\omega:=\d_1u^2-\d_2u^1$ satisfies 
$$
\d_t\omega+u\cdot\nabla\omega+\d_1(\textstyle\frac1\rho)\d_2\Pi-\d_2(\frac1\rho)\d_1\Pi=0,
$$
hence is no  transported by the flow of $u$ if the density is not a constant.
\end{itemize}
Under the assumptions of Theorem \ref{th:main}, the solutions
to \eqref{eq:ddeuler} satisfy the following
Beale-Kato-Majda type continuation  criterion. For simplicity, we state the result for positive times only.
\begin{thm}\label{th:BKM}
  Under the hypotheses of Theorem $\ref{th:main},$
    consider a solution $(\rho,u,\nabla\Pi)$ to $\eqref{eq:ddeuler}$ on $[0,T)\times\R^N$ 
  with the properties described in Theorem $\ref{th:main}.$ 
   If in addition  
   \begin{equation}\label{eq:bu0}
   \int_0^T\bigl(\|\nabla u\|_{L^\infty}+\|\nabla\Pi\|_{B^{s-1}_{p,r}}\bigr)\,dt<\infty
 \end{equation}then $(\rho,u,\nabla\Pi)$ may be continued beyond $T$ into 
a  solution of $\eqref{eq:ddeuler}$ with the same regularity.

Moreover, in the case $s>1+N/p$, the term $\nabla u$ may be replaced by ${\rm curl}
\, u$
in $\eqref{eq:bu0}.$
\end{thm}  
\begin{rem}\label{r:BKM}
The above statement has two important consequences:
\begin{itemize}\item
First, as  Condition $(C)$ implies
that $B^{s-1}_{p,r}$ is embedded in $L^\infty,$ 
one can show by means of an easy bootstrap argument
that for data in $B^s_{p,r},$ the lifespan
of a solution in $B^s_{p,r}$ 
is \emph{the same} as the lifespan in $B^{\frac Np+1}_{p,1}$
(which is the larger space in this scale satisfying Condition $(C)$).
\item Second, by combining the previous remark with an induction argument, 
we see that if we start with smooth data $(a_0,u_0)$  such that
the  derivatives 
at any order of $\nabla a_0$ and $u_0$ are in 
 in $L^p$ then we get a  local-in-time  smooth solution with the same properties.
 In addition, as above,  the lifespan for that smooth solution is only determined by  
 the $B^{\frac Np+1}_{p,1}$ regularity. 
This generalizes prior results in the H\"older spaces framework
 in the case of a bounded domain obtained in \cite{BV3}.
\end{itemize}
\end{rem}
In the two-dimensional case, the  assumption that $u_0\in L^2$  is somewhat 
restrictive since if, say,   the initial vorticity is in the Schwartz class then $u_0\in L^2$ implies
that the vorticity has average~$0$ over $\R^N.$
This motivates the following statement which  allows for \emph{any} suitably  smooth 
initial vector-field with compactly supported   vorticity.   
\begin{thm}\label{th:2}
Let $T_0$ be in $]0,\infty[$ and let $(s,p,r)$ satisfy
 Condition~$(C)$ with  $2\leq p\leq 4.$
Let $u_0$ be a divergence-free vector-field with coefficients
in $B^s_{p,r}.$ 
Assume that $\rho_0$ is positive, bounded and bounded away from zero, and 
that $\nabla\rho_0\in B^{s-1}_{p,r}.$  Finally, suppose that the body force $f$ 
has coefficients in $L^1([T_0,T_0];B^s_{p,r})$ and that the potential part $\cQ f$
of $f$ is in $\cC([-T_0,T_0];L^2).$

There exists a time $T\in(0,T_0]$ such that
 System $\eqref{eq:ddeuler}$ supplemented with initial data
$(\rho_0,u_0)$ has a unique local solution $(\rho,u,\nabla\Pi)$ on $[-T,T]\times\R^N$ with:
\begin{itemize}
\item $\rho^{\pm1}\in\cC_b([-T,T]\times\R^N),\quad
D\rho\in\cC_w([-T,T];B^{s-1}_{p,r}),$
\item $u\in\cC_w([-T,T];B^s_{p,r}),$
\item $\nabla\Pi\in\cC([-T,T];L^2)\cap L^1([-T,T];B^s_{p,r}).$
\end{itemize}
Besides, 
 time continuity holds with respect to the strong topology, if $r<\infty,$
 and the continuation criteria stated in Theorem $\ref{th:BKM}$ also hold under the above  assumptions.
\end{thm}
As regards the well-posedness theory, 
the study of the limit case $p=\infty$ is of interest for different reasons. 
First, the Besov space $B^1_{\infty,1}$ is the largest one for which 
Condition (C) holds. Second, the usual H\"older spaces
belong to the family $B^s_{\infty,r}$ (take $r=\infty$) and 
are suitable for the study of the  
 propagation of tangential regularity in \eqref{eq:ddeuler}, 
and of vortex patches than we plan to do in future works.
 \begin{thm}\label{th:3}
Assume that $u_0\in B^s_{\infty,r}\cap L^p,$
$f\in L^1([-T_0,T_0];B^s_{\infty,r}\cap L^p)$
and $\rho_0\in B^{s}_{\infty,r}$ for  some $p\in(1,\infty)$
and some $s>1$ (or $s\geq1$ if $r=1$).  
There exists a constant $\alpha>0$ depending only on $s$ and $N$ such that
if, for some  positive 
real number $\ov\rho$  we have 
\begin{equation}\label{eq:smallrho}
\|\rho_0-\ov\rho\|_{B^s_{\infty,r}}\leq\alpha\ov\rho,
\end{equation}
then there exists some $T>0$ such that
System $\eqref{eq:ddeuler}$ has a unique solution
$(\rho,u,\nabla\Pi)$ with 
\begin{itemize}
\item $\rho\in\cC([-T,T];B^s_{\infty,r})$
(or $\cC_w([-T,T];B^s_{\infty,r})$ if $r=\infty$),
\item $u\in\cC([-T,T];L^p\cap B^s_{\infty,r}),$
(or $u\in\cC([-T,T];L^p)\cap\cC_w([-T,T];B^s_{\infty,r})$ if $r=\infty$),
\item $\nabla\Pi\in L^1([-T,T];B^s_{\infty,r}).$
\end{itemize}
\end{thm}
Note that our result holds for small perturbations of a constant density state only.
The reason why is that, in contrast with the previous statements, 
here bounding the pressure relies on estimates for the ordinary  Laplace operator 
$\Delta.$ In other words, the heterogeneity $(a-\ov a)\nabla\Pi$ is treated
as a small perturbation term. 
We expect this smallness assumption to be just a technical artifact. However, 
removing it goes beyond the scope of this paper. 


\section{Tools}\label{s:tools}

Our results mostly rely on the use of a nonhomogeneous dyadic partition of unity 
with respect to the Fourier variable, the so-called Littlewood-Paley decomposition. 
More precisely, 
fix a smooth radial function
$\chi$ supported in (say) the ball $B(0,\frac43),$ 
equals to $1$ in a neighborhood of $B(0,\frac34)$
and such that $r\mapsto\chi(r\,e_r)$ is nondecreasing
over $\R^+,$ and set
$\varphi(\xi)=\chi(\frac\xi2)-\chi(\xi).$
\smallbreak
The {\it dyadic blocks} $(\Delta_q)_{q\in\Z}$
 are defined by\footnote{Throughout we agree  that  $f(D)$ stands for 
the pseudo-differential operator $u\mapsto\cF^{-1}(f\cF u).$} 
$$
\dq:=0\ \hbox{ if }\ q\leq-2,\quad\Delta_{-1}:=\chi(D)\quad\hbox{and}\quad
\Delta_q:=\varphi(2^{-q}D)\ \text{ if }\  q\geq0.
$$
We  also introduce the following low frequency cut-off:
$$
S_qu:=\chi(2^{-q}D)=\sum_{p\geq q-1}\Delta_p\quad\text{for}\quad q\geq0.
$$
The following classical properties will be used freely throughout in the paper:
\begin{itemize}
\item for any $u\in\cS',$ the equality $u=\sum_{q}\dq u$ makes sense in $\cS'$;
\item for all $u$ and $v$ in $\cS',$
the sequence
$(S_{q-1}u\,\dq v)_{q\in\N}$ is spectrally
supported in dyadic annuli. Indeed,
as $\Supp\chi\subset B(0,\frac43)$ and $\Supp\varphi\subset\{\xi\in\R^n\,/\,\frac34\leq|\xi|\leq\frac83\},$ we have
$$\textstyle{\Supp\bigl({\cF}(S_{q-1}u\,\dq v)\bigr)\subset
\bigl\{\xi\in\R^N\,/\, \frac1{12}\cdot2^q\leq |\xi|\leq 
\frac{10}3\cdot2^q\bigr\}.}$$
\end{itemize}
One can now define what a Besov space $B^s_{p,r}$ is:
\begin{defi}
\label{def:besov}
  Let  $u$ be a tempered distribution, $s$ a real number, and 
$1\leq p,r\leq\infty.$ We set
$$
\|u\|_{B^s_{p,r}}:=\bigg(\sum_{q} 2^{rqs}
\|\Delta_q  u\|^r_{L^p}\bigg)^{\frac{1}{r}}\ \text{ if }\ r<\infty
\quad\text{and}\quad
\|u\|_{B^s_{p,\infty}}:=\sup_{q} 2^{qs}
\|\Delta_q  u\|_{L^p}.
$$
We then define the space $B^s_{p,r}$ as  the
subset of  distributions $u\in {\cS}'$ such  that
$\|u\|_{B^s_{p,r}}$ is finite.
\end{defi}
The Besov spaces have many interesting properties which will be recalled throughout the paper
whenever they are needed.
For the time being, let us just recall that if Condition $(C)$ holds true
then $B^s_{p,r}$ is an algebra continuously embedded in the set $C^{0,1}$
of bounded Lipschitz functions (see e.g. \cite{BCD}, Chap. 2), and that the gradient operator
maps $B^s_{p,r}$ in $B^{s-1}_{p,r}.$
The following result will be also needed:
\begin{prop}\label{p:CZ}
Let $F$ be a smooth homogeneous function of degree $0$ on $\R^N\setminus\{0\}.$
Then for all $p\in(1,\infty),$ Operator $F(D)$ is a self-map on  $L^p.$
 In addition, if $r\in[1,\infty]$ and $s\in\R$  then $F(D)$
is a self-map on $B^s_{p,r}.$ 
 \end{prop}
 \begin{proof}
 The continuity on $L^p$ stems from the H\"ormander-Mihlin theorem (see e.g. \cite{grafakos}).
 The rest of the proposition follows from the fact that if $u\in B^s_{p,r}$ then one may write, owing
 to $F(2^{-q}\xi)=F(\xi)$ for all $q\geq0$ and $\xi\not=0,$
 $$
 F(D)u=F(D)\Delta_{-1}u+\sum_{q\geq0}(F\tilde\varphi)(2^{-q}D)\dq u$$
 where $\tilde\varphi$ is a smooth function with compact support away from the origin
 and value $1$ on the support of $\varphi.$ Note that ${\cF}^{-1}(F\tilde\varphi)$
 is in $L^1.$ 
 Therefore, the standard convolution inequality implies that 
 $$
 \|(F\tilde\varphi)(2^{-q}D)\dq u\|_{L^p}\leq C\|\dq u\|_{L^p}
 $$
 while the $L^p$ continuity result implies that 
 $$
 \|F(D)\Delta_{-1}u\|_{L^p}\leq \|\Delta_{-1}u\|_{L^p}.
 $$
 Putting these two results together entails that $F(D)$ maps $B^s_{p,r}$ in itself.
  \end{proof}
  \begin{rem}\label{r:CZ}
  Both  the Leray projector $\cP$ over divergence free vector-fields
  and $\cQ:={\rm Id}-\cP$  satisfy the assumptions of the above proposition.
 Indeed, in Fourier variables, we have for all vector-field $u$ with coefficients in $\cS'(\R^N),$
 $$
 \hat{\cQ u}(\xi)=-\frac{\xi}{|\xi|^2}\,\xi\cdot\hat u(\xi).
 $$
  \end{rem}
  The following lemma (referred in what follows as \emph{Bernstein's inequalities})
  describe the way derivatives act on spectrally localized functions.
  \begin{lem}
\label{lpfond}
{\sl
Let  $0<r<R.$   A
constant~$C$ exists so that, for any nonnegative integer~$k$, any couple~$(p,q)$ 
in~$[1,\infty]^2$ with  $q\geq p\geq 1$ 
and any function $u$ of~$L^p$,  we  have for all $\lambda>0,$
$$
\displaylines{
{\rm Supp}\, \widehat u \subset   B(0,\lambda R)
\Longrightarrow
\|D^k u\|_{L^q} \leq
 C^{k+1}\lambda^{k+N(\frac{1}{p}-\frac{1}{q})}\|u\|_{L^p};\cr
{\rm Supp}\, \widehat u \subset \{\xi\in\R^N\,/\, r\lambda\leq|\xi|\leq R\lambda\}
\Longrightarrow C^{-k-1}\lambda^k\|u\|_{L^p}
\leq
\|D^k u\|_{L^p}
\leq
C^{k+1}  \lambda^k\|u\|_{L^p}.
}$$
}
\end{lem}   
  The first Bernstein inequality entails the following embedding result:
  \begin{prop}\label{p:embed}
  The space $B^{s_1}_{p_1,r}$ is embedded in the space $B^{s_2}_{p_2,r}$ whenever
  $$
  1\leq p_1\leq p_2\leq\infty\quad\hbox{and}\quad
  s_2\leq s_1-N/p_1+N/p_2. 
  $$
  \end{prop}
  \begin{rem}\label{r:embed}
  Recall that for all $s\in\R,$ the Besov space $B^s_{2,2}$ coincides
  with the nonhomogeneous Sobolev space $H^s.$
  Furthermore if, for $k\in\N,$  we  denote by $W^{k,p}$ the set of $L^p$ functions
 with derivatives up to order $k$ in $L^p$ then
 we have the following chain of continuous embedding:
 $$
 B^k_{p,1}\hookrightarrow W^{k,p}\hookrightarrow B^k_{p,\infty}.
 $$
 \end{rem}  
 Let us now recall a few nonlinear estimates in Besov spaces. 
 Formally, any product  of two tempered distributions $u$ and $v,$ may be decomposed
into 
\begin{equation}\label{eq:bony}
uv=T_uv+T_vu+R(u,v)
\end{equation}
with 
$$
T_uv:=\sum_qS_{q-1}u\dq v,\quad
T_vu:=\sum_q S_{q-1}v\dq u\ \hbox{ and }\ 
R(u,v):=\sum_q\sum_{|q'-q|\leq1}\dq u\,\Delta_{q'}v.
$$
The above operator $T$ is called ``paraproduct'' whereas
$R$ is called ``remainder''.
The decomposition \eqref{eq:bony} has been introduced by J.-M. Bony in \cite{Bony}.
We shall sometimes use the notation 
$$
T'_uv:=T_uv+R(u,v).
$$
The paraproduct and remainder operators have many nice continuity properties. 
The following ones will be of constant use in this paper (see the proof in e.g. \cite{BCD}, Chap. 2):
\begin{prop}\label{p:op}
For any $(s,p,r)\in\R\times[1,\infty]^2$ and $t<0,$ the paraproduct operator 
$T$ maps $L^\infty\times B^s_{p,r}$ in $B^s_{p,r},$
and  $B^t_{\infty,\infty}\times B^s_{p,r}$ in $B^{s+t}_{p,r}.$
Moreover, the following estimates hold:
$$
\|T_uv\|_{B^s_{p,r}}\leq C\|u\|_{L^\infty}\|\nabla v\|_{B^{s-1}_{p,r}}\quad\hbox{and}\quad
\|T_uv\|_{B^{s+t}_{p,r}}\leq C\|u\|_{B^t_{\infty,\infty}}\|\nabla v\|_{B^{s-1}_{p,r}}.
$$
For any $(s_1,p_1,r_1)$ and $(s_2,p_2,r_2)$ in $\R\times[1,\infty]^2$ such that 
$s_1+s_2>0,$ $1/p:=1/p_1+1/p_2\leq1$ and $1/r:=1/r_1+1/r_2\leq1$
the remainder operator $R$ maps 
$B^{s_1}_{p_1,r_1}\times B^{s_2}_{p_2,r_2}$ in $B^{s_1+s_2}_{p,r}.$
\end{prop}
Combining the above proposition with Bony's decomposition \eqref{eq:bony}, 
we easily get the following ``tame estimate'':
\begin{cor}\label{c:op}
Let $a$ be a bounded function such that $\nabla a\in B^{s-1}_{p,r}$ for some $s>0$
and $(p,r)\in[1,\infty]^2.$  Then for any $b\in B^s_{p,r}\cap L^\infty$ we have $ab\in B^s_{p,r}\cap L^\infty$
and there exists a constant $C$ depending only on $N,$ $p$ and $s$
such that 
$$
\|ab\|_{B^s_{p,r}}\leq \|a\|_{L^\infty}\|b\|_{B^s_{p,r}}+\|b\|_{L^\infty}\|Da\|_{B^{s-1}_{p,r}}.
$$
\end{cor}
The following result pertaining to the composition of functions
in Besov spaces will be needed for estimating the reciprocal  of the density.
\begin{prop}\label{p:comp}
Let $I$ be a bounded interval of $\R$ and $F:I\rightarrow\R$ a smooth function. 
Then for all compact subset $J\subset I,$ $s>0$ and $(p,r)\in[1,\infty]^2$ there exists a constant $C$
such that for all $a\in B^s_{p,r}$ with values in $J,$ we have
$F(a)\in B^s_{p,r}$ and 
$$
\|F(a)\|_{B^s_{p,r}}\leq C\|a\|_{B^s_{p,r}}.
$$
\end{prop}
Our results concerning Equations \eqref{eq:ddeuler} rely strongly on a priori estimates
in Besov spaces for the transport equation
$$
\left\{\begin{array}{l}
\d_tf+v\cdot\nabla f=g,\\[1ex]
f_{|t=0}=f_0.\end{array}\right.\leqno(T)
$$
We shall often use the  following result, the proof of which  may be found in e.g. \cite{BCD}, Chap. 3, 
or in the appendix of \cite{D1}.  
\begin{prop}\label{p:transport}  Let $1\leq p,r\leq\infty$ and 
$\sigma>0.$
Let $f_0\in B^\s_{p,r},$
 $g\in L^1([0,T];B^\s_{p,r})$ and  $v$
 be a time dependent vector-field in $\cC_b([0,T]\times\R^N)$
such that for some $p_1\geq p,$ we have  
$$
\begin{array}{lllll}
\nabla v&\in&  L^1([0,T];B^{\frac N{p_1}}_{p_1,\infty}\cap
L^\infty)&\hbox{\rm if}&\sigma<1+\frac{N}{p_1},\\[1.5ex]
\nabla v&\in& L^1([0,T];B^{\sigma -1}_{p_1,r})
&\hbox{\rm if}&\sigma>1\!+\!\frac{N}{p_1},\quad
\hbox{\rm  or }\  \sigma=1\!+\!\frac{N}{p_1}\ \hbox{\rm  and }\ r=1.
\end{array}
$$
Then Equation $(T)$ has a unique solution  $f$ in 
\begin{itemize}
\item the space $\cC([0,T];B^{\sigma}_{p,r})$ if $r<\infty,$
\item the space $\Bigl(\bigcap_{\sigma'<\sigma} \cC([0,T];B^{\sigma'}_{p,\infty})\Bigr)
\bigcap  \cC_w([0,T];B^{\sigma}_{p,\infty})$ if $r=\infty.$
\end{itemize}
Moreover,  for all $t\in[0,T],$ we have
\begin{equation}\label{sanspertes1}
e^{-CV(t)}\|f(t)\|_{B^\sigma_{p,r}}\leq
\|f_0\|_{B^\sigma_{p,r}}+\int_0^t
e^{-CV(t')}
\|g(t')\|_{B^\sigma_{p,r}}\,dt'
\end{equation} 
$$\displaylines{
\mbox{with}\quad\!\! V'(t):=\left\{
\begin{array}{l}
\!\|\nabla v(t)\|_{B^{\frac N{p_1}}_{p_1,\infty}\cap
L^\infty}\!\!\!\quad\mbox{if}\!\!\!
\quad\sigma<1+\frac{N}{p_1},\\
\!\|\nabla v(t)\|_{B^{\sigma-1}_{p_1,r}}\ \mbox{ if }\
\sigma>1+\frac{N}{p_1},\quad\!\mbox{or }\ 
\sigma=1\!+\!\frac{N}{p_1}\quad\!\!\!\mbox{and}\!\!\!\quad
r=1. 
\end{array}\right.\hfill} $$
 If $f=v$ then, for all $\sigma>0,$   Estimate  $(\ref{sanspertes1})$ holds with
$V'(t):=\|\nabla v(t)\|_{L^\infty}.$
\end{prop}


\section{Elliptic estimates}\label{s:elliptic}

In this section, we want to prove high regularity estimates in Besov spaces for the
following elliptic equation
\begin{equation}\label{eq:elliptic}
-\div(a\nabla\Pi)=\div F\quad\hbox{in }\ \R^N
\end{equation}
where $a=a(x)$ is a given suitably smooth bounded function satisfying 
\begin{equation}\label{eq:ellipticity}
a_*:=\inf_{x\in\R^N}a(x)>0.
\end{equation}
Let us recall that in the case $a\equiv1$ the following result is available:
\begin{prop}\label{p:Lp}
If $a\equiv1$ and $p\in(1,\infty)$ then 
there exists a solution map $F\mapsto\nabla\Pi$ continuous on $L^p.$
\end{prop}
\begin{proof} We set $\nabla\Pi=\nabla(-\Delta)^{-1}\div F.$
Obviously the pseudo-differential operator $\nabla(-\Delta)^{-1}\div$
satisfies the conditions of Proposition \ref{p:CZ}. 
Hence $F\mapsto \nabla\Pi$ is a continuous self-map on $L^p.$ 
\end{proof}
We now turn to the study of \eqref{eq:elliptic} for  nonconstant coefficients.
For the convenience of the reader let us first  establish the following classical result
pertaining to the $L^2$ case.  
\begin{lem}\label{l:laxmilgram}
For all vector-field $F$ with coefficients in $L^2,$ there exists a tempered distribution $\Pi,$
unique up to  constant functions, 
such that  $\nabla\Pi\in L^2$ and  
Equation $\eqref{eq:elliptic}$ is satisfied. 
In addition, we have 
\begin{equation}\label{eq:el0}
a_* \|\nabla\Pi\|_{L^2}\leq\|F\|_{L^2}.
\end{equation}
\end{lem}
\begin{proof} The existence part of the statement is a consequence of the Lax-Milgram theorem.
Indeed, for $\lambda>0,$ consider the following bilinear map:
$$
b_\lambda(u,v)=(a\nabla u\mid \nabla v)_{L^2}
+\lambda(u\mid v)_{L^2}\quad\hbox{for }\ u\ \hbox{ and }\ 
v\ \hbox{ in }\ H^1(\R^N).
$$
Obviously $b_\lambda$ is continuous and coercive, hence, given $F\in (L^2(\R^N))^N,$
 there exists a unique $\Pi_\lambda\in H^1(\R^N)$ so that
 $$
 b_\lambda(u,\Pi_\lambda)=(u\mid F)_{L^2}\quad\hbox{for all } \ u\in H^1(\R^N).
 $$
Taking $u=\Pi_\lambda$ and using the Cauchy-Schwarz inequality, we see that \eqref{eq:el0} is satisfied by $\Pi_\lambda.$ Hence $(\nabla\Pi_\lambda)_{\lambda>0}$ is bounded in
$L^2$ and  there exist some $Q\in (L^2(\R^N))^N$ and a sequence $(\lambda_n)_{n\in\N}$
converging to $0,$ such that 
$\nabla\Pi_{\lambda_n}\rightharpoonup Q$ weakly in $L^2.$ 
Note that this implies that $Q$ satisfies 
$\div(aQ)=\div F$ in the distributional sense, and also that $Q$ is the gradient of
some tempered distribution~$\Pi.$ 
Besides, we have
$$
\|\nabla\Pi\|_{L^2}=\|Q\|_{L^2}\leq\liminf\|\nabla\Pi_{\lambda_n}\|_{L^2}\leq a_*^{-1}\|F\|_{L^2}.
$$
As regards uniqueness, it suffices to check that the constant functions are the only
tempered solutions with gradient in $L^2$ which satisfy   \eqref{eq:elliptic}
with $F\equiv0.$ 
So let us consider  $\Pi\in\cS'$ with $\nabla\Pi\in L^2$
and  $\div(a\nabla\Pi)=0.$
We thus have 
\begin{equation}\label{eq:el00}
\int a\nabla u\cdot\nabla\Pi\,dx=0 \quad\hbox{for all }\ u\in H^1.
\end{equation}
By taking advantage of the Fourier transform and of Parseval equality, it is
easy to check that for $n>0,$ the tempered distribution 
$\Pi_n:=(\Id-\chi(nD))\Pi$ (where the cut-off function $\chi$ has been 
defined in Section \ref{s:tools}) belongs to $H^1.$ Hence one may take $u=\Pi_n$
in \eqref{eq:el00} and we get 
 $$ 
 \int a\nabla\Pi\cdot \nabla\Pi_n\,dx=0 \quad\hbox{for all }\  n>0.
$$
 As $\nabla\Pi_n$ tends to $\nabla\Pi$ in $L^2$ 
  and   $a\geq a^*>0,$ this readily  implies that
 $\nabla\Pi=0.$
\end{proof}
Let us now establish higher order estimates.
\begin{prop}\label{p:elliptic} Let $1<p<\infty$ and $1\leq r\leq\infty.$ Let $a$ be a bounded  function satisfying $\eqref{eq:ellipticity}$ 
and such that $Da\in B^{s-1}_{p,r}$ for some $s>1+N/p$ or $s\geq 1+N/p$ if $r=1.$
\begin{itemize}
\item If $1<p<\infty,$ $\sigma\in(1,s]$ and $\nabla\Pi\in B^\sigma_{p,r}$ satisfies $\eqref{eq:elliptic}$ for some function $F$ such that $\div F\in B^{\sigma-1}_{p,r}$  then we have
for some constant $C$ depending only on $s,\sigma,p,N,$  
$$
a_* 
\|\nabla\Pi\|_{B^\s_{p,r}}\leq C\biggl(\|\div F\|_{B^{\s-1}_{p,r}}+a_*\Bigl(1
+a_*^{-1}\|Da\|_{B^{s-1}_{p,r}}\Bigr)^\s\|\nabla\Pi\|_{L^p}\biggr).$$
\item
If  $2\leq p<\infty$ and  $F$ is in $L^2$ and satisfies $\div F\in B^{\s-1}_{p,r}$
for some  $\s\in(1+N/p-N/2,s]$ then  Equation $\eqref{eq:elliptic}$ has a unique solution 
$\Pi$ (up to constant functions) such that $\nabla\Pi\in L^2\cap B^{\s}_{p,r}.$
Furthermore, Inequality $\eqref{eq:el0}$ is satisfied
and there exists a positive exponent $\gamma$ depending only on $\s,$ $p,$ $N$ and
a  positive constant  $C$ depending only on $s,\s,p,N$ such that
$$
a_* \|\nabla\Pi\|_{B^\s_{p,r}}\leq C\biggl(\|\div F\|_{B^{\s-1}_{p,r}}+\Bigl(1+a_*^{-1}
\|Da\|_{B^{s-1}_{p,r}}\Bigr)^\gamma\|F\|_{L^2}\biggr).$$
\item If $\sigma>1$ and $1<p<\infty$ then the following inequality holds:
$$
a_* \|\nabla\Pi\|_{B^\s_{p,r}}\leq C\Bigl(\|\div F\|_{B^{\s-1}_{p,r}}+
\|\nabla a\|_{L^\infty}\|\nabla\Pi\|_{B^{\s-1}_{p,r}}+\|\nabla\Pi\|_{L^\infty}\|\nabla a\|_{B^{\s-1}_{p,r}}
\Bigr).
$$
\end{itemize}
\end{prop}
\begin{proof}
Throughout, $(c_q)_{q\geq-1}$ denotes a sequence in the unit
sphere of $\ell^r.$

The proof relies on two ingredients:
\begin{enumerate}
\item[{\it (i)}]  the following commutator 
estimates (see Lemmas \ref{l:com}  and \ref{l:combis} in the appendix) 
\begin{eqnarray}\label{eq:comest}
&&\|\div[a,\dq]\nabla\Pi\|_{L^p}\leq 
Cc_q2^{-q(\s-1)}\|\nabla a\|_{B^{s-1}_{p,r}}\|\nabla\Pi\|_{B^{\s-1}_{p,r}},\\\label{eq:comestbis}
&&\|\div[a,\dq]\nabla\Pi\|_{L^p}\leq 
Cc_q2^{-q(\s-1)}\bigl(\|\nabla\Pi\|_{L^\infty}\|\nabla a\|_{B^{\s-1}_{p,r}}
+\|\nabla a\|_{L^\infty}\|\nabla\Pi\|_{B^{\s-1}_{p,r}}\bigr)
\end{eqnarray}
which hold  true whenever $\s\in(0,s]$ and $(s,p,r)$ satisfies Condition~$(C)$
(as regards \eqref{eq:comest})
and whenever $\s>1$ (as concerns \eqref{eq:comestbis});
\item[{\it (ii)}]   a Bernstein type inequality (see Lemma \ref{l:Bernstein}  in the appendix).
\end{enumerate}
For proving the first part of the lemma,  apply the spectral cut-off operator $\dq$ to \eqref{eq:elliptic}.
We get
$$
-\div(a\dq\nabla\Pi)=\div\dq F+\div([\dq,a] \nabla\Pi)\quad\hbox{for all }\ q\geq0.
$$
Hence, multiplying both sides by $|\dq\Pi|^{p-2}\dq\Pi$ and integrating over $\R^N,$ we get
$$\displaylines{
-\int|\dq\Pi|^{p-2}\dq\Pi\,\div(a\dq\nabla\Pi)\,dx=
\int|\dq\Pi|^{p-2}\dq\Pi\,\div\dq F\,dx\hfill\cr\hfill+\int|\dq\Pi|^{p-2}\dq\Pi\,\div([\dq,a]\nabla\Pi)\,dx.}
$$
Apply  Lemma \ref{l:Bernstein} to bound by below the
left-hand side of the above inequality. Using H\"older's inequality to handle the right-hand side, 
we get for all $q\geq0,$
\begin{equation}\label{eq:el2}
a_*2^{2q}\|\dq\Pi\|_{L^p}^p\leq C\|\dq\Pi\|_{L^p}^{p-1}
\Bigl(\|\div\dq F\|_{L^p}+\|\div[\dq,a]\nabla\Pi\|_{L^p}\Bigr).
\end{equation}
To deal with the last term, one may now take advantage of Inequality \eqref{eq:comest}.
Since, for $q\geq 0,$ we have
 $\|\dq\nabla\Pi\|_{L^p}\approx2^q\|\dq\Pi\|_{L^p}$ according to Lemma \ref{lpfond},
we get
after our multiplying Inequality \eqref{eq:el2} by $2^{q(\sigma-1)}$:
$$
a_*2^{q\s}\|\dq\nabla\Pi\|_{L^p}\leq C
\Bigl(2^{q(\s-1)}\|\dq\div F\|_{L^p}+c_q\|\nabla a\|_{B^{s-1}_{p,r}}\|\nabla\Pi\|_{B^{\s-1}_{p,r}}\Bigr)\quad\hbox{for all }\ q\in\N.
$$
Taking the $\ell^r$ norm of both sides and adding up the low frequency
block pertaining to $\Delta_{-1}\nabla\Pi,$ we get
\begin{equation}\label{eq:el1}
a_*\|\nabla\Pi\|_{B^\s_{p,r}}\leq C\Bigl(\|\div F\|_{B^{\s-1}_{p,r}}
+\|\nabla a\|_{B^{s-1}_{p,r}}\|\nabla\Pi\|_{B^{\s-1}_{p,r}}+a_*\|\Delta_{-1}\nabla\Pi\|_{L^p}\Bigr).
\end{equation}
Observe that
 $\|\Delta_{-1}\nabla\Pi\|_{L^p}\leq C\|\nabla\Pi\|_{L^p},$
and that  the following interpolation inequality is available (recall that $0<\sigma-1$):
$$
\|\nabla\Pi\|_{B^{\s-1}_{p,r}}\leq C\|\nabla\Pi\|_{L^p}^{\frac1\s}
\|\nabla\Pi\|_{B^\s_{p,r}}^{1-\frac1\s}.
$$
Then, applying a suitable Young inequality completes the proof of the first part of the proposition.
\smallbreak
Let us now tackle the proof of the second part of the Proposition. 
As $F\in L^2,$ the existence of a solution $\nabla\Pi$ in $L^2$
is ensured by Lemma \ref{l:laxmilgram}.
Let us admit for a while that $\nabla\Pi\in B^\sigma_{p,r}$ and let us prove
the desired inequality. 
  As $p\geq2,$ 
we have
$$
L^2\hookrightarrow B^{N(\frac1p-\frac12)}_{p,\infty}.
$$
Hence, as $B^{\s-1}_{p,r}$ is an interpolation space between $B^{N(\frac1p-\frac12)}_{p,\infty}$
and $B^\s_{p,r}$ (here comes the assumption that $\s-1>N/p-N/2$), one may write
for some convenient exponent $\theta=\theta(p,\s,N)\in(0,1),$
$$
\|\nabla\Pi\|_{B^{\s-1}_{p,r}}\leq C\|\nabla\Pi\|_{L^2}^{\theta}
\|\nabla\Pi\|_{B^\s_{p,r}}^{1-\theta}.
$$
In addition, as $p\geq2,$ Bernstein's inequality implies that 
$$\|\Delta_{-1}\nabla\Pi\|_{L^p}\leq C\|\nabla\Pi\|_{L^2}.
$$
Hence, plugging the last two   inequalities in  \eqref{eq:el1} and using \eqref{eq:el0} yields
$$
a_*\|\nabla\Pi\|_{B^\s_{p,r}}\leq C\Bigl(\|\div F\|_{B^{\s-1}_{p,r}}+\|F\|_{L^2}
+a_*^{-1}\|F\|_{L^2}^\theta\|\nabla a\|_{B^{s-1}_{p,r}}
\bigl(a_*\|\nabla\Pi\|_{B^{\s}_{p,r}}\bigr)^{1-\theta}\Bigr).
$$
Then applying Young's inequality completes the proof.

Remark that Inequality \eqref{eq:el1} remains valid whenever $\nabla\Pi$ is in $B^{\sigma-1}_{p,r}.$
Starting from the fact that the constructed solution $\nabla\Pi$ is in 
$B^{N(\frac1p-\frac12)}_{p,\infty},$ a straightforward induction argument 
allows to state that  $\nabla\Pi$ is indeed in  $B^{\sigma}_{p,r}.$
This completes the second part of the proof. 
\smallbreak
For proving the last part of the proposition, the starting point is Inequality \eqref{eq:el2}
which implies that
$$
a_*2^{q\s}\|\nabla\dq\Pi\|_{L^p}\leq C2^{q(\s-1)}\bigl(\|\dq\div F\|_{L^p}+\|\div[\dq,a]\nabla\Pi\|_{L^p}\bigr).
$$
Now, taking advantage of Inequality \eqref{eq:comestbis} then summing up over $q\geq-1,$
we readily obtain the desired result.
\end{proof}


\section{Proof of the first local well-posedness result}\label{s:th:main}

As a preliminary step, let us observe that System \eqref{eq:ddeuler} is \emph{time reversible}. 
That is, changing  $(t,x)$ in $(-t,-x)$
restricts the  study of the Cauchy problem to the evolution for \emph{positive} times.
To simplify the presentation,  we shall thus concentrate from now on to the unique solvability 
of the system for positive times only. 

In the first part of this section, we establish the uniqueness part of Theorem \ref{th:main}.
When proving existence, it is convenient to treat the two cases $p\geq2$ and $p<2$ separately. 
The reason why is that the proof strongly relies on Proposition \ref{p:elliptic} 
which enables to compute the pressure \emph{only if $p\geq2$}. 
Indeed, if $p<2$ then only an a priori estimate is stated. 

So, in the second part of this section, we prove the existence in the case $p\geq2.$
The third subsection is devoted to the proof of Theorem \ref{th:BKM} in the case $p\geq2.$
It will be needed for proving the existence part of Theorem \ref{th:main}
in the case $p<2.$  
The following   part of this section is  devoted to the proof of Theorems 
\ref{th:main} and \ref{th:BKM} in the case $p<2.$ 
In the last paragraph,  we justify the claim pertaining to the case $p>N/(N-1)$
 (see just after the statement of Theorem \ref{th:main}).
\smallbreak
For expository purpose, we shall assume in this section and in the rest of the
paper that $r<\infty.$ For treating  the case $r=\infty,$ it is only a matter of replacing the strong topology
by weak topology  whenever  regularity 
\emph{up to index} $s$ is involved.


\subsection{Uniqueness}

Uniqueness in Theorems \ref{th:main} is a  consequence
of the following general stability  result for solutions to \eqref{eq:ddeuler}.
\begin{prop}\label{p:uniqueness}
Let $(\rho_1,u_1,\nabla\Pi_1)$ and $(\rho_2,u_2,\nabla\Pi_2)$
satisfy $\eqref{eq:ddeuler}$ with exterior forces $f_1$ and $f_2.$ Assume in addition 
that $\rho_1$ and  $\rho_2$ are bounded and bounded away from zero, 
that $\du:=u_2-u_1$ and $\dr:=\rho_2-\rho_1$ belong to 
$\cC^1([0,T];L^2),$ that $\df:=f_2-f_1$ is in $\cC([0,T];L^2)$
 and that $\nabla\Pi_1,$ $\nabla\rho_1$ 
and $\nabla u_1$ belong to $L^1([0,T];L^\infty).$
Then for all $t\in[0,T],$ we have \begin{equation}\label{eq:uniq}
\|\dr(t)\|_{L^2}+\|(\sqrt\rho_2\du)(t)\|_{L^2}
\leq e^{A(t)}
\biggl(\|\dr(0)\|_{L^2}+\|(\sqrt\rho_2\du)(0)\|_{L^2}
+\int_0^te^{-A(\tau)}\|\sqrt\rho_2\df\|_{L^2}\biggr)
 \end{equation}
 with $\displaystyle{A(t):=\int_0^t
\biggl(\Bigl\|\frac{\nabla\rho_1} {\sqrt\rho_2}\Bigr\|_{L^\infty}+\Big\|\frac{\nabla\Pi_1}{\rho_1\sqrt\rho_2}\Big\|_{L^\infty}+\|\nabla u_1\|_{L^\infty}\biggr)\,d\tau.}$
\end{prop}
\begin{proof}
On the one hand, as
$$\d_t\dr+u_2\cdot\nabla\dr=-\du\cdot\nabla\rho_1,$$ taking the $L^2$ inner product
with $\dr$ and integrating by parts in the second term of the left-hand side yields
\begin{equation}\label{stab1}
\|\dr(t)\|_{L^2}\leq  \|\dr(0)\|_{L^2}+ \int_0^t \|(\sqrt\rho_2\du)\|_{L^2}
\Bigl\|\frac{\nabla\rho_1} {\sqrt\rho_2}\Bigr\|_{L^\infty}\,d\tau.
\end{equation}
On the other hand, denoting  $\nabla\dPi:=\nabla\Pi_2-\nabla\Pi_1,$ we notice that
$$
\rho_2(\d_t\du+u_2\cdot\nabla\du)+\nabla\dPi
=\rho_2\biggl(\df+\frac{\dr}{\rho_1\rho_2}\nabla\Pi_1-\du\cdot\nabla u_1\biggr).$$
So taking the $L^2$ inner product of the second equation with 
$\du,$ integrating by parts and using the fact that $\div\du=0$ and that
$$
\d_t\rho_2+u_2\cdot\nabla\rho_2=0,
$$
we eventually get 
$$
\|(\sqrt\rho_2\du)(t)\|_{L^2}
\leq \|(\sqrt\rho_2\du)(0)\|_{L^2}+ \int_0^t\!\biggl(\!\|\sqrt\rho_2\df\|_{L^2}
+\|\dr\|_{L^2}\Bigl\|\frac{\nabla\Pi_1}{\rho_1\sqrt\rho_2}\Bigr\|_{L^\infty}\!+\|\nabla u_1\|_{L^\infty}\|\sqrt\rho_2\du\|_{L^2}\!\biggr)d\tau.
 $$
Adding up Inequality \eqref{stab1} to 
the above inequality and applying Gronwall lemma
completes the proof of the proposition.
\end{proof}
\smallbreak\noindent{\it Proof of uniqueness in Theorem \ref{th:main}.}
Consider two solutions $(\rho_1,u_1,\nabla\Pi_1)$ and
$(\rho_2,u_2,\nabla\Pi_2)$ of \eqref{eq:ddeuler} with the same data.
Under the assumptions of Theorem \ref{th:main}, 
it is clear that the velocity and pressure fields  
satisfy the assumptions of the above proposition. 
As concerns the density, we notice that $u_i\in\cC([0,T];L^2)$ and $\nabla\rho_i\in\cC([0,T];L^\infty)$
for $i=1,2$ 
implies that $\d_t\rho_i\in\cC([0,T];L^2).$ Hence we have $\dr\in\cC^1([0,T];L^2).$
Therefore Inequality \eqref{eq:uniq} implies that 
$(\rho_1,u_1,\nabla\Pi_1)\equiv(\rho_2,u_2,\nabla\Pi_2)$ on $[0,T]\times\R^N.$
This completes the proof of the uniqueness in Theorem \ref{th:main}.


\subsection{The proof of existence in Theorem \ref{th:main}: the case $2\leq p<\infty$}
\label{ss:main}

We notice that, formally, the density-dependent incompressible Euler equations 
are equivalent to\footnote{Recall that $\cP$ stands for the Leray projector over
divergence free vector-fields.} 
\begin{equation}\label{eq:ddeuler1}
\left\{\begin{array}{l}
\d_ta+u\cdot\nabla a=0\quad\hbox{with} \ a:=1/\rho,\\[1ex]
\d_tu+u\cdot\nabla u+a\nabla\Pi=f,\\[1ex]
-\div(a\nabla\Pi)= \div(u\cdot\nabla\cP u)-\div f.
\end{array}\right.
\end{equation}
Let us give conditions under which this equivalence is rigorous. 
\begin{lem}\label{l:equiv}
Let $u$ be a time-dependent vector-field with coefficients
in $\cC^1([0,T]\times\R^N)$ and such that $\cQ u\in\cC^1([0,T];L^2).$
Assume that $\nabla\Pi\in\cC([0,T];L^2).$
Let $\rho$ be a continuous bounded function on $[0,T]\times\R^N$
which is positive and bounded away from $0.$ 

If in addition $\div u(0,\cdot)\equiv0$  in $\R^N$
then $(\rho,u,\nabla\Pi)$ is a solution to $\eqref{eq:ddeuler}$
if and only if $(a,u,\nabla\Pi)$ satisfies $\eqref{eq:ddeuler1}.$
 \end{lem}
 \begin{proof} If $(\rho,u,\nabla\Pi)$ satisfies \eqref{eq:ddeuler}
 then, owing to $\rho>0,$ we see that $a:=1/\rho$ satisfies the
 first equation of \eqref{eq:ddeuler1}. 
 Next,  applying Operator $\div$ to the velocity equation of \eqref{eq:ddeuler} divided by
 $\rho,$ and using that $\cP u=u$ yields the third equation of \eqref{eq:ddeuler1}. 
 \smallbreak
 Conversely, if $(a,u,\nabla\Pi)$ satisfies \eqref{eq:ddeuler1}, it is obvious, owing to positivity, 
 that $\rho:=1/a$ satisfies the density equation of \eqref{eq:ddeuler}. 
 In   order to justify  that the other two equations are  satisfied, 
 it is only a matter of proving that  $\div u\equiv0.$
 For that, one may 
 apply $\cQ$  to the second equation. Then, using the third equation, we discover that
 $$\d_t\cQ u+\cQ(u\cdot\nabla\cQ u)=0.$$
 Recall that  $\cQ u\in\cC^1([0,T];L^2).$ 
 Therefore, taking the $L^2$ inner product with $\cQ u,$ we get
 $$
 \frac12\frac d{dt}\|\cQ u\|_{L^2}^2+\bigl(\cQ(u\cdot\nabla\cQ u)\mid \cQ u\bigr)_{L^2}=0.
 $$
 As  $\cQ^T=\cQ$ and $\cQ^2=\cQ,$  we thus get after  integrating by parts in the second term:
 $$
  \frac d{dt}\|\cQ u\|_{L^2}^2=\int |\cQ u|^2\div u\,dx,
  $$
  and, as $\cQ u(0,\cdot)=0,$ Gronwall lemma ensures that $\cQ u\equiv0.$
  Hence $\div u=0.$
\end{proof}

As explained in Lemma \ref{l:equiv}, it suffices to solve System \eqref{eq:ddeuler1}.
So, for $T>0,$ let us introduce the set $E_T$ of functions $(a,u,\nabla\Pi)$ such that 
$$
\begin{array}{lll}a\in\cC_b([0,T]\times\R^N),&&
\nabla a\in\cC([0,T];B^{s-1}_{p,r}),\\[1ex]
 u\in\cC^1([0,T];L^2)\cap \cC([0,T];B^{s}_{p,r}),&&
\nabla\Pi\in\cC([0,T];L^2)\cap L^1([0,T];B^{s}_{p,r}).\end{array}
$$
We denote $$a_*:=\inf_{x\in\R^N}a_0(x),\quad  
 a^*:=\sup_{x\in\R^N}a_0(x),\quad \rho_*:=\inf_{x\in\R^N}\rho_0(x)\ \hbox{ and }\  
 \rho^*:=\sup_{x\in\R^N}\rho_0(x).$$
 
 Note that if $\rho$ is bounded and bounded away from zero, and satisfies
 $\nabla\rho\in B^{s-1}_{p,r}$ then the same properties hold for $a$
 (and conversely). This may be easily shown by combining
 Propositions \ref{p:embed} and \ref{p:comp}. Moreover, there
 exists some constant $C$ depending only on $a_*,$ $a^*,$ $N$  and on the regularity
 parameters such that
 $$
 C^{-1}\|\nabla\rho\|_{B^{s-1}_{p,r}}\leq  \|\nabla a\|_{B^{s-1}_{p,r}}
 \leq C\|\nabla\rho\|_{B^{s-1}_{p,r}}.
 $$
 This fact will be used repeatedly in the rest of the paper. 
  \subsubsection*{Step 1. Construction  of a sequence of approximate  solutions}
As a first step for solving \eqref{eq:ddeuler1},  we  construct a sequence  
$(a^n,u^n,\nabla\Pi^n)_{n\in\N}$ of global approximate solutions 
which belong to $E_T$ for all $T>0.$
 \smallbreak
 For doing so, one may argue by induction. We first set  $(a^0,u^0,\nabla\Pi^0):=(a_0,u_0,0).$
Next, we assume  that $(a^n,u^n,\nabla\Pi^n)$ has been constructed over $\R^+,$
belongs to the space $E_T$ for all $T>0$ and that there exists a positive time
$T^*$ such that for all $t\in[0,T^*],$
\begin{eqnarray}\label{eq:loinduvide}
&a_*\leq a^{n}(t,x)\leq a^*,\\
\label{eq:bound4}
&\|\nabla a^{n}(t)\|_{B^{s-1}_{p,r}}\leq  2\|\nabla a_0\|_{B^{s-1}_{p,r}}\quad\hbox{for all }\ 
t\in[0,T^*],\\\label{eq:energy0}
&\|\sqrt{\rho^n(t)}u^n(t)\|_{L^2}\leq \sqrt{\rho^*a^*}\bigl(4\|\sqrt{\rho_0}\,u_0\|_{L^2}
+8\sqrt{\rho^*}\|f\|_{L^1_t(L^2)}\bigr)\quad\hbox{with }\ 
\rho^n:=1/a^n,\\\label{eq:hyp1}
&U^n(t)\leq 4U_0(t)+C_0\rho^*A_0\|\div f\|_{L^1_t(B^{s-1}_{p,r})}
+C_0(\rho^*A_0)^{\gamma+1}\bigl(\|u_0\|_{L^2}+\|f\|_{L^1_t(L^2)}\bigr),\\\label{eq:hyp2}
&a_*\|\nabla\Pi^n\|_{L^1_t(B^s_{p,r})}\leq C\biggl(\Int_0^t(U^n(\tau))^2\,d\tau
+\|\div f\|_{L_t^1(B^{s-1}_{p,r})}\hspace{3cm}\nonumber\\
&\hspace{5cm}+(\rho^*A_0)^\gamma
\bigl(\|u_0\|_{L^2}+\|f\|_{L_t^1(L^2)}\bigr)\biggr),\\\label{eq:hyp3}
&\|\nabla\Pi^n\|_{L_t^1(L^2)}\,d\tau\leq\sqrt{\rho^*}\|\sqrt{\rho_0}\,u_0\|_{L^2}
+3\rho^*\|f\|_{L^1_t(L^2)}\end{eqnarray}
with   $A_0:=a^*+\|Da_0\|_{B^{s-1}_{p,r}},$ 
$U_0(t):=\|u_0\|_{B^s_{p,r}}+\|f\|_{L^1_t(B^s_{p,r})}$ and  
$U^n(t):=\|u^n(t)\|_{B^{s}_{p,r}}.$
The positive exponent $\gamma$ is given by Proposition \ref{p:elliptic}. 
The constants $C_0$ and $C$ depend only on $(s,p,r)$ and $N,$
and may be made explicit from the following computations
(in fact one can take $C_0=2C^2$ with $C$ large enough).
\medbreak
Denoting by $\psi^n$ the flow of $u^n,$ (which belongs to 
$\cC^1(\R^+\times\R^N)$ owing to $u^n\in\cC(\R^+;B^s_{p,r})$
and to $B^{s-1}_{p,r}\hookrightarrow\cC_b$), 
we set 
$$a^{n+1}(t,x):=a_0((\psi_t^n)^{-1}(x))\quad\hbox{and}\quad
\rho^{n+1}(t,x):=\rho_0((\psi_t^n)^{-1}(x)).
$$
As $\psi_t^n$ is a diffeomorphism over $\R^N$ for all $t\geq0,$
we  have  
$$
\|a^{n+1}(t)\|_{L^\infty}=\|a_0\|_{L^\infty}=a^*\quad\hbox{and}\quad
\|\rho^{n+1}(t)\|_{L^\infty}=\|\rho_0\|_{L^\infty}=\rho^*.
$$
Hence \eqref{eq:loinduvide} is satisfied by $a^{n+1}.$ 
 In addition, we have 
 $$
 \d_ta^{n+1}+u^n\cdot\nabla a^{n+1}=0
 $$
 so that  for all $i\in\{1,\cdots,N\},$ 
$$
\d_t\d_ia^{n+1}+u^n\cdot\nabla\d_ia^{n+1}=-\d_iu^n\cdot\nabla a^{n+1}.
$$
As $\d_ia^{n+1}_{|t=0}=\d_ia_0\in B^{s-1}_{p,r}$ by assumption, 
(a slight generalization of) Proposition \ref{p:transport} 
combined with Gronwall lemma guarantees that $\nabla a^{n+1}\in\cC(\R^+;B^{s-1}_{p,r})$
and that
\begin{equation}\label{eq:bound1}
\|\nabla a^{n+1}(t)\|_{B^{s-1}_{p,r}}\leq  e^{C\int_0^tU^n(\tau)\,d\tau}\|\nabla a_0\|_{B^{s-1}_{p,r}}.
\end{equation}
So if we assume that $T^*$ has been chosen so that
\begin{equation}\label{eq:time1}
C\int_0^{T^*}U^n(t)\,dt\leq\log2
\end{equation}
then $a^{n+1}$ satisfies \eqref{eq:bound4}.\smallbreak
Next,  we want to define $u^{n+1}$ as the unique solution  in $\cC(\R^+;B^s_{p,r})$ of the
 transport equation: 
\begin{equation}\label{eq:un}
\d_tu^{n+1}+u^n\cdot\nabla u^{n+1}=-a^{n+1}\nabla\Pi^{n}+f,\qquad u^{n+1}_{|t=0}=u_0.
\end{equation}
That  the right-hand side belongs to $L^1_{loc}(\R^+;B^s_{p,r})$ is a consequence of
Corollary \ref{c:op} and of the embedding $B^{s-1}_{p,r}\hookrightarrow L^\infty.$
In addition, we have for a.e. positive time
\begin{equation}\label{eq:bound3a}
\|a^{n+1}\nabla\Pi^n\|_{B^s_{p,r}}\leq C\bigl(\|a^{n+1}\|_{L^\infty}+\|\nabla a^{n+1}\|_{B^{s-1}_{p,r}}\bigr)
\|\nabla \Pi^n\|_{B^{s}_{p,r}}.
\end{equation}
So finally, the existence of $u^{n+1}\in\cC(\R^+;B^s_{p,r})$ is ensured by 
Proposition \ref{p:transport}, and we have
\begin{eqnarray}\label{eq:bound2}
&&\|u^{n+1}(t)\|_{B^{s}_{p,r}}\leq  e^{C\int_0^tU^n(\tau)\,d\tau}\biggl(\|u_0\|_{B^{s}_{p,r}}
\hspace{3cm}\nonumber\\
&&\qquad+\int_0^te^{-C\int_0^\tau U^n(\tau')\,d\tau'}\Bigl(\bigl(\|a^{n+1}\|_{L^\infty}
+\|\nabla a^{n+1}\|_{B^{s-1}_{p,r}}\bigr)\|\nabla \Pi^n\|_{B^{s}_{p,r}}
+\|f\|_{B^s_{p,r}}\Bigr)\,d\tau\biggr).
\end{eqnarray}
Therefore, if we restrict our attention to those $t$ that are in $[0,T^*]$
with $T^*$ satisfying \eqref{eq:time1},  we see that for all $t\in[0,T^*],$
$$
U^{n+1}(t)\leq  2U_0(t)
+CA_0
\int_0^t\|\nabla\Pi^n\|_{B^s_{p,r}}\,d\tau\quad\hbox{with }\ 
A_0:=a^*+\|\nabla a_0\|_{B^{s-1}_{p,r}}.
$$
So if we assume that $T^*$ and $C_0$ have been chosen so that
\begin{equation}\label{eq:time2}
2C^2\rho^*A_0\int_0^{T^*}U^n(t)\,dt\leq1\quad\hbox{and }\ C_0=2C^2
\end{equation}
then taking advantage of Inequalities \eqref{eq:hyp2} and \eqref{eq:hyp1}, 
we see that $u^{n+1}$ satisfies \eqref{eq:hyp1} on $[0,T^*].$

Let us now prove \eqref{eq:energy0} for $u^{n+1}.$
First, we notice  that the right-hand side of \eqref{eq:un} belongs to $\cC(\R^+;L^2)$ so that
$u^{n+1}$ is in $\cC^1(\R^+;L^2).$ 
As $\rho^{n+1}$ is bounded and 
$\cC^1$ with respect to the time and space variables, this allows us to  take the $L^2$ inner product of the equation
for $u^{n+1}$ with $\rho^{n+1}u^{n+1}$. We readily get
\begin{equation}\label{eq:energy1}
\frac12\frac d{dt}\|\sqrt{\rho^{n+1}}\,u^{n+1}\|_{L^2}^2-\int\rho^{n+1}u^{n+1}\cdot f\,dx=\frac12\int\rho^{n+1}|u^{n+1}|^2\div u^n\,dx
-\bigl(\nabla\Pi^n\mid u^{n+1}\bigr)_{L^2}.
\end{equation}
Let us point out that  $u^n$ and $u^{n+1}$  \emph{need not} be divergence-free, so 
that the right-hand side may be nonzero. 
However, from  the above inequality, it is easy to get
$$\displaylines{
\|(\sqrt{\rho^{n+1}}\,u^{n+1})(t)\|_{L^2}\leq
\|{\sqrt\rho_0}\,u_0\|_{L^2}\hfill\cr\hfill+\int_0^t\biggl(\sqrt{a^*}\|\nabla\Pi^n\|_{L^2}
+\sqrt{\rho^*}\|f\|_{L^2}
+\frac12\|\sqrt{\rho^{n+1}}u^{n+1}\|_{L^2}\|\div u^n\|_{L^\infty}\biggr)\,dx.}
$$
So, if we assume that $C$ has been taken large enough in \eqref{eq:time1}
then Gronwall's lemma implies that 
\begin{equation}\label{eq:bound5a}
\|(\sqrt{\rho^{n+1}}\,u^{n+1})(t)\|_{L^2}\leq 2\bigl(\|{\sqrt\rho_0}\,u_0\|_{L^2}
+\sqrt{\rho^*}\|f\|_{L_t^1(L^2)}+\sqrt{a^*}\|\nabla\Pi^n\|_{L_t^1(L^2)}\bigr).
\end{equation}
Now, putting the above inequality together with Inequality \eqref{eq:hyp3}
ensures that Inequality \eqref{eq:energy0} is also satisfied by $u^{n+1}$
on $[0,T^*].$
 \smallbreak
 To finish with, we have to construct  the approximate pressure $\Pi^{n+1}.$ 
For that, we aim at solving  the following  
elliptic equation
\begin{equation}\label{eq:pres}
\div(a^{n+1}\nabla\Pi^{n+1})=\div(f-u^{n+1}\cdot\nabla\cP u^{n+1})
\end{equation}
for every positive time. 

We have already proved that $a^{n+1}$ satisfies the required ellipticity condition
through \eqref{eq:loinduvide}.
Moreover, as $u^{n+1}\in\cC(\R^+;B^s_{p,r}),$ Remark \ref{r:CZ} ensures that 
$\nabla\cP u^{n+1}$ is in $\cC(\R^+;B^{s-1}_{p,r}).$ As   
$B^{s-1}_{p,r}\hookrightarrow L^\infty$ and $u^{n+1}\in\cC(\R^+;L^2),$  we thus have 
$u^{n+1}\cdot\nabla\cP u^{n+1}\in \cC(\R^+;L^2)$  and 
$$\begin{array}{lll}
\|u^{n+1}\cdot\nabla\cP u^{n+1}\|_{L^2}&\leq& \sqrt{a^*}
\bigl\|\sqrt{\rho^{n+1}}u^{n+1}\bigr\|_{L^2}\|\nabla\cP u^{n+1}\|_{L^\infty},\\[1ex]
&\leq& C \sqrt{a^*}\bigl\|\sqrt{\rho^{n+1}}\,u^{n+1}\bigr\|_{L^2}
\|\nabla u^{n+1}\|_{B^{s-1}_{p,r}}.\end{array}
$$
Therefore Lemma \ref{l:laxmilgram} guarantees that \eqref{eq:pres} has
a solution $\nabla\Pi^{n+1}$ in $\cC(\R^+;L^2)$ which satisfies
\begin{equation}\label{eq:bound6}
a_*\|\nabla\Pi^{n+1}\|_{L_t^1(L^2)}\leq \|f\|_{L^1_t(L^2)}
+C\sqrt{a^*}\int_0^tU^{n+1}\|\sqrt{\rho^{n+1}}u^{n+1}\|_{L^2}\,d\tau.
\end{equation}
Let us insert Inequality \eqref{eq:bound5a} in the above inequality. 
We see that  if  $T^*$ has been chosen so that 
\begin{equation}\label{eq:time4}
4Ca^*\rho^*\int_0^{T^*}U^{n+1}\,d\tau\leq1
\end{equation}
then Inequality \eqref{eq:bound6} implies that
$$\|\nabla\Pi^{n+1}\|_{L_t^1(L^2)}\leq \frac32\rho^*\|f\|_{L^1_t(L^2)}
+\frac12\sqrt{\rho^*}\|\sqrt{\rho_0}u_0\|_{L^2}
+\frac12\|\nabla\Pi^{n}\|_{L_t^1(L^2)},
$$
hence Inequality \eqref{eq:hyp3} is satisfied by $\nabla\Pi^{n+1}$ on $[0,T^*].$
\smallbreak
In order to prove that $\nabla\Pi^{n+1}$ belongs to $L^1_{loc}(\R^+;B^s_{p,r}),$
one may apply   the second part of Proposition~\ref{p:elliptic}.
Indeed, because,  owing to $\div\cP u^{n+1}=0,$ we have 
$$\div(u^{n+1}\cdot\nabla\cP u^{n+1})=
\nabla u^{n+1}:\nabla\cP u^{n+1}$$ 
and  as $B^{s-1}_{p,r}$ is an algebra, the term 
$\div(u^{n+1}\cdot\nabla\cP u^{n+1})$ is in $B^{s-1}_{p,r}$ and 
$$
\|\div(u^{n+1}\cdot\nabla\cP u^{n+1})\|_{B^{s-1}_{p,r}}\leq C (U^{n+1})^2.$$
Hence  Proposition \ref{p:elliptic}  implies that for all $t\in\R^+,$
  $$\displaylines{
  a_*\|\nabla\Pi^{n\!+\!1}\|_{L_t^1(B^s_{p,r})}\leq
 C\biggl(\int_0^t(U^{n\!+\!1})^2\,d\tau\,+\|\div f\|_{L_t^1(B^{s-1}_{p,r})}\hfill\cr\hfill
 +\bigl(1+\rho*\|Da^{n\!+\!1}\|_{L_t^\infty(B^{s-1}_{p,r})}\bigr)^\gamma
\bigl(\|f\|_{L_t^1(L^2)}+\|u^{n\!+\!1}\cdot\nabla\cP u^{n+1}\|_{L_t^1(L^2)}\bigr)\Bigr),}
  $$
  whence, using \eqref{eq:bound4} at rank $n+1$ and H\"older inequality,  we get
    $$\displaylines{
  a_*\|\nabla\Pi^{n\!+\!1}\|_{L_t^1(B^s_{p,r})}\leq
 C\biggl(\int_0^t(U^{n\!+\!1})^2\,d\tau\,+\|\div f\|_{L_t^1(B^{s-1}_{p,r})}\hfill\cr\hfill
 +(\rho^*A_0)^\gamma
\bigl(\|f\|_{L_t^1(L^2)}+\sqrt{a^*}\|\sqrt{\rho^{n+1}}u^{n+1}\|_{L^\infty_t(L^2)}
\int_0^tU^{n\!+\!1}\,d\tau\bigr)\biggr).}
  $$
  Taking advantage of Inequality \eqref{eq:energy0} at rank $n+1$ one can now conclude that
  if \eqref{eq:time4} holds then $\nabla\Pi^{n+1}$ satisfies \eqref{eq:hyp2}. 
  
  At this stage we have proved that if Inequalities \eqref{eq:loinduvide} to 
  \eqref{eq:hyp3} hold for $(a^n,u^n,\nabla\Pi^n)$ then they also 
  hold for $(a^{n+1},u^{n+1},\nabla\Pi^{n+1})$
  \emph{provided $T^*$ satisfies Inequalities \eqref{eq:time1}, 
  \eqref{eq:time2}  and \eqref{eq:time4}}. Note that \eqref{eq:time2} is the strongest condition. 
  Obviously it is satisfied if we set 
  \begin{equation}\label{eq:time}
  T^*:=\sup\Bigl\{t>0\,/\, \rho^*t A_0\Bigl(U_0(t)+\rho^*A_0\|\div f\|_{L^1_t(B^{s-1}_{p,r})}
+(\rho^*A_0)^{\gamma+1}
  \bigl(\|u_0\|_{L^2}+\|f\|_{L^1_t(L^2)}\bigr)\Bigr)\leq c\Bigr\}
  \end{equation}  
  for a  small  enough constant $c$ depending only on $s,$ $p$ and $N.$


  \subsubsection*{Step 2 . Convergence of the sequence}
    
Let $\tilde a^n:=a^n-a_0.$ In this step, we shall establish that $(\tilde a^n,u^n,\nabla\Pi^n)_{n\in\N}$
is a Cauchy sequence in $\cC([0,T^*];L^2).$

Let $\da^n:=\tilde a^{n+1}-\tilde a^n,$
$\du^n:=u^{n+1}-u^n$ and $\dPi^{n}:=\Pi^{n+1}-\Pi^n.$
We have for $n\geq2,$
\begin{equation}\label{eq:cauchy}\left\{\begin{array}{l}
\d_t\da^n+u^n\cdot\nabla\da^n=-\du^{n-1}\cdot\nabla a^n,\\[1ex]
\d_t\du^n+u^n\cdot\nabla\du^n=-\du^{n-1}\cdot\nabla u^n-a^n\nabla\dPi^{n-1}
-\da^n\nabla\Pi^n,\\[1ex]
\div(a^{n-1}\nabla\dPi^{n-1})=-\div\bigl(\du^{n-1}\cdot\nabla\cP u^n
+u^{n-1}\cdot\nabla\cP\du^{n-1}+\da^{n-1}\nabla\Pi^n\bigr).\end{array}
\right.
\end{equation}
For all $n\in\N,$ we have $\d_t\tilde a^{n+1}=-u^n\cdot\nabla a^{n+1}.$
So, given that, according to the previous step,  $u^n\in\cC([0,T^*];L^2)$
and $\nabla a^{n+1}\in\cC_b([0,T^*]\times\R^N),$
and that $\tilde a^{n+1}_{|t=0}=0,$ 
we discover that 
$\tilde  a^{n+1},$ and thus also $\da^n,$ are in $\cC^1([0,T^*];L^2).$
Taking the $L^2$ inner product of the equation for $\da^n$ with $\da^n,$
we thus get
$$
\frac12\frac d{dt}\|\da^n\|_{L^2}^2=\frac12\int (\da^n)^2\div u^n\,dx-\int \du^{n-1}\cdot\nabla a^n
\,\da^n\,dx,
$$
whence for all $t\in[0,T^*],$
\begin{equation}\label{eq:cauchy1}
\|\da^n(t)\|_{L^2}\leq \frac12\int_0^t\|\div u^n\|_{L^\infty}\|\da^n\|_{L^2}\,d\tau
+\int_0^t\|\nabla a^n\|_{L^\infty}\|\du^{n-1}\|_{L^2}\,d\tau.
\end{equation}
Next, taking the $L^2$ inner product of the equation for $\du^n$ with $\rho^{n+1}\du^n,$
 performing integration by parts and using the equation for $\rho^{n+1},$ we get
$$
\frac12\frac d{dt}\int\!\rho^{n\!+\!1}|\du^n|^2\,dx=\frac12\int\! \rho^{n+1}|\du^n|^2\div u^n\,dx
-\int\! \rho^{n\!+\!1}\du^{n}\cdot\bigl(\du^{n-1}\cdot\nabla u^n+a^n\nabla\dPi^{n-1}+\da^n\nabla\Pi^n\bigr)\,dx.
$$
Hence
\begin{eqnarray}
&&\|\sqrt{\rho^{n+1}(t)}\du^n(t)\|_{L^2}\leq 
\frac12\Int_0^t\|\div u^n\|_{L^\infty}\|\sqrt{\rho^{n+1}}\du^n\|_{L^2}\,d\tau\nonumber\\\label{eq:cauchy2}
&&\hspace{2cm}+\Int_0^t\bigl(\|\nabla u^n\|_{L^\infty}\|\du^{n-1}\|_{L^2}
+\|a^n\|_{L^\infty}\|\nabla\dPi^{n-1}\|_{L^2}
+\|\nabla\Pi^n\|_{L^\infty}\|\da^n\|_{L^2}\bigr)\,d\tau.
\end{eqnarray}
Adding up Inequalities \eqref{eq:cauchy1} and \eqref{eq:cauchy2}, applying Gronwall lemma and 
using the fact that $\rho^{n+1}\geq\rho_*$ and the bounds stated in the first step, 
we thus get for all  $t\in[0,T^*],$
\begin{equation}\label{eq:cauchy3}
\|(\da^n,\du^n)(t)\|_{L^2}\leq C_{T^*}\biggl(\int_0^t\|(\da^{n-1},\du^{n-1})(\tau)\|_{L^2}\,d\tau
+\int_0^t\|\nabla\dPi^{n-1}(\tau)\|_{L^2}\,d\tau\biggr),
\end{equation}
where the constant $C_{T^*}$ depends only on $T^*$ and on the initial data. 
\smallbreak
In order to bound $\nabla\dPi^{n-1},$ we shall use that for any $C^1$ vector-fields $a$ and $b,$ we have
$$
\div(a\cdot\nabla b)=\div(b\cdot\nabla a)+\div(a\,\div b)-\div (b\,\div a).
$$
Applying this to $a=u^{n-1}$ and $b=\cP \du^{n-1}$ and bearing in mind  that 
$\div\cP\du^{n-1}=0,$
we deduce from the third equation of \eqref{eq:cauchy} that
$$
\div(a^{n-1}\nabla\dPi^{n-1})=\div\bigl(\cP\du^{n-1}\div u^{n-1}-\cP\du^{n-1}\cdot\nabla u^{n-1}
-\du^{n-1}\cdot\nabla\cP u^{n}-\da^{n-1}\nabla\Pi^n\bigr).
$$
Therefore, Lemma \ref{l:laxmilgram} and the fact that $\|\cP\|_{\cL(L^2;L^2)}=1$ guarantee that 
$$
a_*\|\nabla\dPi^{n-1}\|_{L^2}
\leq\|\du^{n-1}\|_{L^2}\bigl(\|\div u^{n-1}\|_{L^\infty}
+\|\nabla u^{n-1}\|_{L^\infty}+\|\nabla\cP u^n\|_{L^\infty}\bigr)
+\|\da^{n-1}\|_{L^2}\|\nabla\Pi^n\|_{L^\infty}.
$$
Using  the uniform bounds of the previous step, we thus get for all $t\in[0,T^*],$
\begin{equation}\label{eq:cauchy4}
\|\nabla\dPi^{n-1}\|_{L^2}
\leq C_{T^*}\bigl(\|\du^{n-1}\|_{L^2}+\|\da^{n-1}\|_{L^2}\bigr).
\end{equation}
Plugging Inequality \eqref{eq:cauchy4} in Inequality \eqref{eq:cauchy3}, we end up with (up to a change of $C_{T^*}$), 
$$
\|(\da^n,\du^n)(t)\|_{L^2}\leq C_{T^*}\int_0^t\|(\da^{n-1},\du^{n-1})(\tau)\|_{L^2}\,d\tau.
$$
Arguing by induction, one may  conclude that 
$$
\sup_{t\in[0,T^*]}\|(\da^n,\du^n)(t)\|_{L^2}\leq \frac{(C_{T^*}T^*)^n}{n!}
\sup_{t\in[0,T^*]}\|(\da^0,\du^0)(t)\|_{L^2}.
$$
It is now obvious that  both  $(\tilde  a^n)_{n\in\N}$ and $(u^n)_{n\in\N}$ are Cauchy sequences in 
$\cC([0,T^*];L^2),$ hence converge to some functions $\tilde a$ and $u$
in $\cC([0,T^*];L^2).$ 
Taking advantage of  \eqref{eq:cauchy4}, it is also clear that  $(\nabla\Pi^n)_{n\in\N}$ 
 converges to some function $\nabla\Pi$
in $\cC([0,T^*];L^2).$ 


\subsubsection*{Step 3. Final checking}

Let  $a:=a_0+\tilde a.$
We now  have to check that $(a,u,\nabla\Pi)$ is indeed a solution to \eqref{eq:ddeuler} 
and that it has the properties stated in Theorem \ref{th:main}. 
{}From the previous step, we already know that $(a-a_0),$ $u$ and $\nabla\Pi$ 
are in $\cC([0,T^*];L^2).$ Moreover: 
\begin{itemize}
\item As $(\nabla a^n)_{n\in\N}$ 
is bounded in $L^\infty([0,T^*]; B^{s-1}_{p,r})$ and as Besov spaces
have the Fatou property, 
we  deduce that $\nabla a$ belongs to $L^\infty([0,T^*]; B^{s-1}_{p,r}).$
Since  $(a^n)_{n\in\N}$ is bounded in $L^\infty([0,T^*]\times\R^N),$
we also have $a\in L^\infty([0,T^*]\times\R^N).$
\item 
As  $(u^n)_{n\in\N}$ 
is bounded in $L^\infty([0,T^*]; B^{s}_{p,r}),$
we deduce that $u\in L^\infty([0,T^*]; B^{s}_{p,r}).$
\item Finally, as $(\nabla\Pi^n)_{n\in\N}$ 
is bounded in $L^1([0,T^*]; B^{s}_{p,r})$
we deduce that $\nabla\Pi$ belongs to $L^1([0,T^*]; B^{s}_{p,r}).$
\end{itemize}
Arguing by interpolation, we see that the above sequences converge 
strongly in every intermediate space between $\cC([0,T^*];L^2)$
and $\cC([0,T^*];B^s_{p,r})$ which is more than enough to pass
to the limit in the equations satisfied by $(a^n,u^n,\nabla\Pi^n).$
Hence $(a,u,\nabla\Pi)$ satisfies \eqref{eq:ddeuler1}.

 Passing to the limit in \eqref{eq:energy1}, we see that, in addition, $(\rho,u)$ satisfies
 the energy equality \eqref{eq:energy}. 
 
 Finally, the continuity properties of the solution with respect to the time may be recovered
 by using the equations satisfied by $(a,u,\nabla\Pi),$ and Proposition \ref{p:transport}. 
 
  
 \subsection{A continuation criterion}
 
 The key to the proof of Theorem \ref{th:BKM} 
  is the following lemma:
 \begin{lem}\label{l:continuation1}
  Let $(s,p,r)$ satisfy Condition $(C)$ with $1<p<\infty.$  Consider a solution 
  $(\rho,u,\nabla\Pi)$
 to $\eqref{eq:ddeuler}$ on $[0,T[\times\R^N$ such that\footnote{With the usual 
 convention if $r=\infty$}  $u\in\cC([0,T);B^s_{p,r})$ and
 $$\rho_*\leq\rho\leq\rho^*,\quad\rho\in\cC([0,T)\times\R^N)\ \hbox{ and }\ 
  \nabla\rho\in \cC([0,T);B^{s-1}_{p,r}).$$
  If in addition 
 \begin{equation}\label{eq:blowup}
 \int_0^T\bigl(\|\nabla u\|_{L^\infty}+\|\nabla\Pi\|_{B^{s-1}_{p,r}}\bigr)\,dt<\infty
\end{equation}
then
$$
\int_0^t\|\nabla\Pi\|_{B^s_{p,r}}\,d\tau+\sup_{0\leq t<T}\bigl(\|u(t)\|_{B^s_{p,r}}
+\|\nabla\rho\|_{B^{s-1}_{p,r}}\bigr)<\infty.
$$
 \end{lem}
 \begin{proof}
 Note that $a:=1/\rho$ satisfies the same assumptions as $\rho.$
 Therefore we shall rather work with $a,$ for convenience.
 Recall that
 \begin{equation}\label{eq:nablaa}
 \d_t\d_ka+u\cdot\nabla\d_ka=-\d_ku\cdot\nabla a\quad\hbox{for }\ k=1,\cdots,N.
\end{equation}
So, applying Operator $\dq$ to the above equality and
using that $\div u=0,$ one may write (with the summation convention)
$$
\d_t\dq\d_ka+u\cdot\nabla\dq\d_ka=-\dq(\d_ku\cdot\nabla a)
+\d_j[u^j,\dq]\d_ka.
$$Therefore for all $t\in[0,T),$
\begin{equation}\label{eq:bu1}
\|\dq\d_ka(t)\|_{L^p}\leq \|\dq\d_ka_0\|_{L^p}
+\int_0^t\|\dq(\d_ku\cdot\nabla a)\|_{L^p}\,d\tau+\int_0^t\|\d_j[u^j,\dq]\d_ka\|_{L^p}\,d\tau.
 \end{equation}
According to Proposition \ref{p:op}, the term $\d_ku\cdot\nabla a$ belongs to
$B^{s-1}_{p,r}$ and satisfies
$$
\|\d_ku\cdot\nabla a\|_{B^{s-1}_{p,r}}\leq C\bigl(\|\d_ku\|_{L^\infty}\|\nabla a\|_{B^{s-1}_{p,r}}
+\|\nabla a\|_{L^\infty}\|\d_ku\|_{B^{s-1}_{p,r}}\bigr)
$$
while Lemma \ref{l:combis} ensures that for all $q\geq-1,$
$$
\|\d_j[u^j,\dq]\d_ka\|_{L^p}\leq Cc_q2^{q(s-1)}
\bigl((\|\d_ka\|_{L^\infty}\|\nabla u\|_{B^{s-1}_{p,r}}
+\|\nabla u\|_{L^\infty}\|\d_ka\|_{B^{s-1}_{p,r}}\bigr).
$$
Using the definition of the norm in $B^{s-1}_{p,r},$ we thus get after summation 
in \eqref{eq:bu1} that
\begin{equation}\label{eq:bu2}
\|\nabla a(t)\|_{B^{s-1}_{p,r}}\leq \|\nabla a_0\|_{B^{s-1}_{p,r}}
+C\int_0^t\bigl(\|\nabla u\|_{L^\infty}\|\nabla a\|_{B^{s-1}_{p,r}}
+\|\nabla a\|_{L^\infty}\|\nabla u\|_{B^{s-1}_{p,r}}\bigr)\,d\tau.
\end{equation}
In order to bound the velocity, let us apply the last part of Proposition \ref{p:transport}
to the velocity equation, and the following inequality (which stems from Corollary \ref{c:op}):
$$
\|a\nabla\Pi\|_{B^s_{p,r}}\leq C\bigl(a^*\|\nabla\Pi\|_{B^s_{p,r}}+\|\nabla\Pi\|_{L^\infty}
\|\nabla a\|_{B^{s-1}_{p,r}}\bigr).
$$
We get  for all $t\in[0,T),$
\begin{eqnarray}\label{eq:bu3}
 &&\|u(t)\|_{B^s_{p,r}}\leq e^{C\int_0^t\|\nabla u\|_{L^\infty}\,d\tau}\biggl(
 \|u_0\|_{B^s_{p,r}}\hspace{5cm}
 \nonumber\\&&\hspace{2cm}+\int_0^t e^{-C\int_0^\tau\|\nabla u\|_{L^\infty}\,d\tau'}
\Bigl(\|f\|_{B^s_{p,r}}+Ca^*\|\nabla\Pi\|_{B^s_{p,r}}
+C\|\nabla\Pi\|_{L^\infty}\|\nabla a\|_{B^{s-1}_{p,r}}\Bigr)\,d\tau\biggr).
\end{eqnarray}
In order to bound the pressure term, one may use the fact that 
$$
\div(a\nabla\Pi)=\div f-\div(u\cdot\nabla u)
$$
and apply the last  part of Proposition \ref{p:elliptic}. Performing a time integration
and using the fact that 
$$
\|\div(u\cdot\nabla u)\|_{B^{s-1}_{p,r}}
=\|\nabla u:\nabla u\|_{B^{s-1}_{p,r}}
\leq C\|\nabla u\|_{L^\infty}\|\nabla u\|_{B^{s-1}_{p,r}},
$$
 we get 
$$\displaylines{
a_*\|\nabla\Pi\|_{L_t^1(B^s_{p,r})}\leq 
C\biggl(\|\div f\|_{L_t^1(B^{s-1}_{p,r})}\hfill\cr\hfill
+\int_0^t\bigl(\|\nabla u\|_{L^\infty}\|\nabla u\|_{B^{s-1}_{p,r}}+
\|\nabla a\|_{L^\infty}\|\nabla \Pi\|_{B^{s-1}_{p,r}}
+\|\nabla\Pi\|_{L^\infty}\|\nabla a\|_{B^{s-1}_{p,r}}\bigr)\,d\tau.}
$$
Let us insert this latter inequality in \eqref{eq:bu3}. 
Then adding up Inequality \eqref{eq:bu2} and applying Gronwall lemma
we end up with 
\begin{eqnarray}\label{eq:bu4}&&\|\nabla a(t)\|_{B^{s-1}_{p,r}}
+\|u(t)\|_{B^{s}_{p,r}}
\leq C\exp\biggl(\int_0^t\|(\nabla a,\nabla u,\nabla\Pi)\|_{L^\infty}\,d\tau\biggr)
\hspace{3cm}\nonumber\\&&\hspace{3cm}
\biggl(\|\nabla a_0\|_{B^{s-1}_{p,r}}
+\|u_0\|_{B^{s}_{p,r}}
+\|f\|_{L_t^1(B^{s}_{p,r})}+\int_0^t\|\nabla a\|_{L^\infty}\|\nabla\Pi\|_{B^{s-1}_{p,r}}\,d\tau\biggr)
\end{eqnarray}
for some constant $C$ depending only on the regularity parameters and
 on $N,$ $a_*$ and $a^*.$
\smallbreak
Now, let us notice that $\nabla a$ is bounded on $[0,T)\times\R^N.$
Indeed, from Equation \eqref{eq:nablaa} and Gronwall lemma, we see that
$$
\|\nabla a(t)\|_{L^\infty}\leq e^{\int_0^t\|\nabla u\|_{L^\infty}}\|\nabla a_0\|_{L^\infty}.
$$
As $\nabla\Pi$ is in $L^1([0,T);B^{s-1}_{p,r})$ and $\nabla u$ is in $L^1([0,T);L^\infty)$
 by assumption 
and as $B^{s-1}_{p,r}\hookrightarrow L^\infty,$ we 
discover that both the last term in \eqref{eq:bu4} and  the exponential term  are bounded
on $[0,T).$ 
This completes the proof of the lemma.
\end{proof}
The following lemma implies the first part of Theorem \ref{th:BKM} in the case $p\geq2.$
\begin{lem}\label{l:continuation2}
  Let $(s,p,r)$ satisfy Condition $(C)$ with $2\leq p<\infty.$  Consider a solution 
  $(\rho,u,\nabla\Pi)$
 to $\eqref{eq:ddeuler}$ on $[0,T[\times\R^N$ such that\footnote{With the usual 
 convention if $r=\infty$} 
 \begin{itemize}
 \item $\rho_*\leq\rho\leq\rho^*,$ $\rho\in\cC([0,T)\times\R^N)$
 and $\nabla\rho\in \cC([0,T);B^{s-1}_{p,r}),$
 \item $u\in\cC([0,T);B^s_{p,r})\cap\cC^1([0,T];L^2),$
 \item $\nabla\Pi\in\cC([0,T);L^2)\cap L^1([0,T);B^s_{p,r}).$
 \end{itemize}
 If in addition Condition $\eqref{eq:blowup}$ is satisfied
 then $(\rho,u,\nabla\Pi)$ may be continued beyond $T$ into 
a solution of $\eqref{eq:ddeuler}$ with the above regularity. 
 \end{lem}
\begin{proof}
Lemma \ref{l:continuation1} ensures that $\|u\|_{L^\infty_T(B^s_{p,r})}$
and $\|Da\|_{L^\infty_T(B^{s-1}_{p,r})}$ are finite. 
So one may set 
$$\displaylines{
\eps:=c(\rho^*A_0)^{-1}\Bigl(U_0(T)+\rho^*A_0\|\div f\|_{L^1_T(B^{s-1}_{p,r})}
+(\rho^*A_0)^{\gamma+1}
  \bigl(\|u_0\|_{L^2}+\|f\|_{L^1_T(L^2)}\bigr)\Bigr)^{-1}}
$$
where $c$ is the small constant (depending only on $N$ and $(s,p,r)$) defined in \eqref{eq:time}.
\smallbreak
Then we know from the proof of Theorem \ref{th:main} in the case $p\geq2$ that
 for any $T'<T,$
System~\eqref{eq:ddeuler} with data $(\rho(T'),u(T'),f(T'+\cdot))$ has a unique
solution up to time $\eps.$
Taking $T'=T-\eps/2$ we thus get a continuation of $(\rho,u,\nabla\Pi)$ up to time
$T+\eps/2.$
 \end{proof}
 Let us now justify the last part of Theorem \ref{th:BKM}. 
 It stems from  the following logarithmic interpolation inequality (see e.g. \cite{KOT}):
 $$
 \|\nabla u\|_{L^\infty}\leq C\Bigl(1+\|\nabla u\|_{\dot B^0_{\infty,\infty}}
 \log\bigl(e+\|\nabla u\|_{B^{s-1}_{p,r}}\bigr)\Bigr)\quad\hbox{with }\ 
 \|\nabla u\|_{\dot B^0_{\infty,\infty}}:=\sup_{q\in\Z}\|\varphi(2^{-q}D)\nabla u\|_{L^\infty}
 $$
 which holds true whenever the embedding of $B^{s-1}_{p,r}$ is not critical 
 (that is $s>1+N/p$).
 
Then,  arguing exactly as in  Proposition 5.3 of \cite{D4},  we discover that 
Condition  \eqref{eq:blowup} may be replaced by the following \emph{weaker} condition:
 \begin{equation}\label{eq:blowup1}
 \int_0^T\bigl(\|\nabla u\|_{\dot B^0_{\infty,\infty}}+\|\nabla\Pi\|_{B^{s-1}_{p,r}}\bigr)\,dt<\infty.
\end{equation}
Now, it is classical (see e.g. \cite{BCD}, Chap. 7) that 
 there  exists some constant $C$ such that
$$
\|\nabla u\|_{\dot B^0_{\infty,\infty}}\leq C\|{\rm curl}\,u\|_{L^\infty}.
$$
This completes the proof of  Theorem \ref{th:BKM} in the case $p\geq2.$   


\subsection{The  case $1<p<2$}

Note that by virtue of Proposition \ref{p:embed}, the data  satisfy the assumptions 
of the theorem for  the triplet $(s-N/p+N/2,2,r).$ Hence, applying the theorem in the case $p=2$
supplies a local solution with the $B^{s-N/p+N/2}_{2,r}$ regularity. 
However, proving that the $B^s_{p,r}$ regularity is also preserved, is not utterly obvious. 
For proving that, we shall proceed as follows:
\begin{enumerate}
\item[{\it i)}] first, we  smooth  out the data so as to get a solution in $H^\infty:=\cap_\sigma H^\sigma$
for which the $B^s_{p,r}$ regularity is also preserved;
\item[{\it ii)}] second,  we establish  uniform bounds in $B^s_{p,r}$ on a fixed suitably small time interval;
\item[{\it iii)}] third, we  show the convergence of the sequence of smooth
solutions and that the limit has the required properties.
\end{enumerate}
\subsubsection*{Step 1: smooth solutions} Set
$$
a_0^n:=S_na_0,\quad
u_0^n:=S_nu_0\ \hbox{ and }\ 
 f^n:=S_n f
 $$
 where $S_n$ is the low frequency cut-off introduced in Section \ref{s:tools}.
 
Note that for all large enough $n\in\N$ and $t\in[0,T_0],$  we have 
\begin{eqnarray}\label{eq:Bound1}
&a_*/2\leq a_0^n\leq 2a^*,\quad
\|Da_0^n\|_{B^{s-1}_{p,r}}\leq C\|Da_0\|_{B^{s-1}_{p,r}},\\\label{eq:Bound2}
&\|u_0^n\|_{B^s_{p,r}}\leq C\|u_0\|_{B^s_{p,r}},\\\label{eq:Bound3}
&\|f^n\|_{L^1_t(B^s_{p,r})}\leq C\|f\|_{L^1_t(B^s_{p,r})}\quad\hbox{and}\quad
\|f^n(t)\|_{L^2}\leq \|f(t)\|_{L^2}.
\end{eqnarray}
It is also clear that (with obvious notation) $\nabla a_0^n$ and $u_0^n$ 
are in $B^\infty_{p,r}$ (hence also in $H^\infty$) 
and that $f^n\in\cC([0,T_0];H^\infty)\cap L^1([0,T_0];B^\infty_{p,r}).$

Finally, taking advantage of Lebesgue's dominated convergence theorem one may prove, if $r<\infty,$ that\footnote{Recall that $(\rho_0-\ov\rho)\in L^{p^*}$ by assumption} 
$$\begin{array}{l}
Da_0^n\longrightarrow Da_0\ \hbox{ in }\  B^{s-1}_{p,r}
\quad\hbox{and}\quad (a_0^n-\ov a)\rightarrow (a_0-\ov a)\ \hbox{ in }\ L^{p^*}\ 
\hbox{ with }\ \ov a:=1/\ov\rho,\\[1ex]
u_0^n\longrightarrow u_0\ \hbox{ in }\  B^{s-1}_{p,r},\\[1ex]
f^n\longrightarrow f \hbox{ in }\ 
L^1([0,T_0];B^s_{p,r})\cap \cC([0,T_0];L^2).
\end{array}
$$
As usual, the strong convergence has to be replaced by the weak convergence if $r=\infty.$
\smallbreak
Applying Theorem \ref{th:main}  in the case $p=2$ and using the fact that the lifespan
does not depend on the index of regularity (see Remark \ref{r:BKM}),
 we get a local maximal solution
$(a^n,u^n,\nabla\Pi^n)$ with 
 $Da^n,$ $u^n$ and $\nabla\Pi^n$ in $\cC([0,T_n^*);H^\infty),$ and
\begin{equation}\label{eq:ldv}
a_*/2\leq a^n\leq 2a^*.
\end{equation}
Note that as $a^n$ and $\rho^n$ are  just transported by the (smooth) flow of $u^n,$ we also have
\begin{equation}\label{eq:Bound4}
\|(\rho^n(t)-\ov\rho)\|_{L^{p^*}}=\|\rho_0^n-\ov\rho\|_{L^{p^*}}\leq
C\|\rho_0-\ov\rho\|_{L^{p^*}}\quad\hbox{for all }\ t\in[0,T_n^*)
\end{equation}
(and similarly for $a^n$) 
and $\nabla a^n$ and $\nabla\rho^n$ belong to $\cC([0,T_n^*);B^{s-1}_{p,r}).$
\smallbreak 
Let us now establish that $\nabla\Pi^n$ is in $L^1([0,T];B^s_{p,r})$
for all $T\in[0,T_n^*).$ Fix some $T\in[0,T_n^*).$
Applying Operator $\div$ to the momentum equation of \eqref{eq:ddeuler} and using
that $\div u^n=0$ yields
\begin{equation}\label{eq:Bound5}
\Delta\Pi^n=\div(\rho^n f^n)-\div(\rho^nu^n\cdot\nabla u^n)-\nabla\rho^n\cdot\d_tu^n.
\end{equation}
According to Proposition \ref{p:CZ}, $F\mapsto D^2\Pi$ is a self-map on  $B^s_{p,r}.$ 
Hence,  in order to show that $\nabla\Pi^n\in L^1([0,T];B^s_{p,r}),$ it
suffices to establish that $\nabla\Pi^n\in L^1([0,T];L^p)$ and that 
$\Delta\Pi^n\in L^1([0,T];B^{s-1}_{p,r}).$

Let us first show that  all the terms of the right-hand side of \eqref{eq:Bound5}
are in $L^1([0,T];B^{s-1}_{p,r}).$
Since, by assumption, $f^n\in L^1([0,T];B^s_{p,r})$
and as it as been established that $\rho^n\in L^\infty$ and $\nabla\rho^n\in B^{s-1}_{p,r},$
Corollary \ref{c:op} implies that 
$\div(\rho^n f^n)\in L^1([0,T];B^{s-1}_{p,r}).$
For the next term, we use that for all $i\in\{1,\cdots,N\},$  
$$
(\rho^nu^n\cdot\nabla u^n)^i=\sum_{j} T'_{\rho^n(u^n)^j}\d_ju^n+T_{\d_j(u^n)^i}\rho^n (u^n)^j.
$$
By embedding, $\rho^n u^n$ and $\nabla u^n$ are in $L^{p^*}$ (recall 
that $p^*>2>p$) and, arguing as for $\rho^nf^n,$ one can check 
that $\rho^nu^n$ is in $H^\infty.$ Of course, $\nabla u^n$ is also in $H^\infty.$
Given that $1/p=1/p^*+1/2,$ continuity results
 for the paraproduct and remainder  in the spirit of Proposition \ref{p:op} (see \cite{RS}) 
 ensure that $\rho^n u^n\cdot\nabla u^n$ is in $B^{s-1}_{p,r}.$

For the last term in \eqref{eq:Bound5}, one may write that
$$
\nabla\rho^n\cdot\d_tu^n=T'_{\d_tu^n}\cdot\nabla\rho^n+T_{\nabla\rho^n}\cdot\d_tu^n.
$$
As, by embedding, $\d_tu^n\in L^\infty,$ and as $\nabla\rho^n\in B^{s-1}_{p,r},$ 
continuity results for the paraproduct ensure that the first term in the right-hand side is in 
$B^{s-1}_{p,r}.$
Concerning the second term, one may use that $\nabla\rho^n\in L^{p^*}$ (by embedding)
and that $\d_tu^n\in H^\infty$ (from the equation). 
Hence $\Delta\Pi^n$ is indeed in $L^1([0,T];B^{s-1}_{p,r}),$ as claimed above.

In order to establish  that  $\nabla\Pi^n\in L^1([0,T];L^p),$ we use the fact that, owing to
 $\div\d_tu^n=0,$ one may write 
\begin{equation}\label{eq:Bound6}
\Delta\Pi^n=\div\Bigl(\rho^n f^n-\rho^nu^n\cdot\nabla u^n-(\rho^n-\ov\rho)\d_tu^n\Bigr).
\end{equation}
Hence, it suffices to check that $\rho^n f^n,$ $\rho^n u^n\cdot\nabla u^n$
and $(\rho^n-\ov\rho)\d_tu^n$ are in $L^1([0,T];L^p).$
For $\rho^n u^n$ this is obvious as, by embedding, $f^n\in L^1([0,T];L^p)$
and $u^n\in\cC_b([0,T]\times\R^N).$
By embedding, we also have $\nabla u^n\in\cC([0,T];L^\infty)$
and $u^n\in\cC([0,T];L^p),$ hence 
$\rho^n u^n\cdot\nabla u^n$ is in  $L^1([0,T];L^p).$

To deal with the last term in \eqref{eq:Bound6} the property that 
$(\rho^n-\ov\rho)\in \cC([0,T];L^{p^*})$  comes into play. Indeed,  from the velocity equation, as
the solution is in $H^\infty$, one easily gathers that
$\d_tu^n$ belongs to $\cC([0,T];L^2).$ Hence H\"older's inequality (note that $1/2+1/p^*=1/p$) ensures that $(\rho^n-\ov\rho)\d_tu^n\in\cC([0,T];L^p).$
\smallbreak
To finish this step, one has to prove that $u^n$ is in $\cC([0,T_n^*);B^s_{p,r}).$
In fact, from the product laws in Besov spaces and the  properties of regularity
that have been just established for the pressure and the density, 
we get
$$
\d_tu^n+u^n\cdot\nabla u^n=f^n-a^n\nabla\Pi^n\ \in L^1_{loc}([0,T_n^*);B^s_{p,r}).
$$
As $u_0^n\in B^s_{p,r},$ Proposition \ref{p:transport} ensures 
that~$u^n\in \cC([0,T_n^*);B^s_{p,r}).$


\subsubsection*{Step 2: Uniform estimates}

Let us remark that, by Sobolev embedding and owing to 
\eqref{eq:Bound1}, \eqref{eq:Bound2}, \eqref{eq:Bound3},
 one may find some index $\sigma>d/2+1$
such that $(Da_0^n)_{n\in\N},$
$(u_0^n)_{n\in\N}$ and $(f^n)_{n\in\N}$ are bounded in 
$H^{\sigma-1},$ $H^\sigma$ and $\cC([0,T_0];L^2)\cap L^1([0,T_0];H^\sigma),$
respectively. 
Taking advantage of Theorem \ref{th:main} in the case $p=2$
and of the lower bound provided by \eqref{eq:time} 
we thus deduce that there exists some time $T>0$ 
and some $M>0$ such that for all $n\in\N,$ we have $T_n^*>T$ and 
\begin{equation}\label{eq:Bound7}
\|\nabla a^n\|_{L^\infty_T(H^{\sigma-1})}
+\|u^n\|_{L^\infty_T(H^{\sigma})}
+\|\nabla\Pi^n\|_{L^1_T(H^\sigma)}\leq M.
\end{equation}
Of course the energy equality \eqref{eq:energy} is satisfied on $[0,T]$
by any solution $(a^n,u^n,\nabla\Pi^n).$ 
Recall that in addition, according to the previous step of the proof, 
\eqref{eq:Bound4} is satisfied and 
$$
\nabla a^n\in\cC([0,T];B^{s-1}_{p,r}),\quad
u^n\in\cC([0,T];B^s_{p,r})\ \hbox{ and }\ 
\nabla\Pi^n\in L^1([0,T];B^s_{p,r}).
$$
We claim that, up to a change of $T,$ the norm of the solution 
may be bounded \emph{independently of~$n$} in the space $E_T$ defined
in Subsection \ref{ss:main}. 
In all that follows, we denote by $C_M$ a ``constant'' depending
only on $(s,p,r,N,a_*,a^*)$ and on $M.$

{}From Proposition \ref{p:transport}, we have
\begin{equation}\label{eq:Bound9}
\|\nabla a^n(t)\|_{B^s_{p,r}}\leq \|\nabla a^n_0\|_{B^s_{p,r}}
e^{C\int_0^t\|\nabla u^n\|_{B^{s-1}_{p,r}}\,d\tau}
\end{equation}
and, arguing as for proving Inequality \eqref{eq:bound3a},
\begin{eqnarray}\label{eq:Bound10}
&&\|u^n(t)\|_{B^s_{p,r}}\leq
e^{C\int_0^t\|\nabla u^n\|_{B^{s-1}_{p,r}}\,d\tau}\biggl(\|u^n_0\|_{B^s_{p,r}}\nonumber
\hspace{1cm}\\&&\hspace{1cm}+\int_0^te^{-C\int_0^\tau\|\nabla u^n\|_{B^{s-1}_{p,r}}\,d\tau'}
\Bigl(\|f^n\|_{B^s_{p,r}}+\bigl(a^*+\|\nabla a^n\|_{B^{s-1}_{p,r}}\bigr)\|\nabla\Pi^n\|_{B^s_{p,r}}
\Bigr)d\tau\biggr).
\end{eqnarray}
In order to bound $\nabla\Pi^n,$ we apply the first part of Proposition \ref{p:elliptic} to 
the following equation:
$$
\div(a^n\nabla\Pi^n)=\div\bigl(f^n -u^n\cdot\nabla u^n\bigr).
$$
Using the fact that $B^{s-1}_{p,r}$ is an algebra and the  relation 
$\div(u^n\cdot\nabla u^n)=\nabla u^n:\nabla u^n,$  we end up with
\begin{equation}\label{eq:Bound11}
a_*\|\nabla\Pi^n\|_{L^1_t(B^s_{p,r})}\leq C\biggl(\|f^n\|_{L^1_t(B^s_{p,r})}
+\int_0^t\|u^n\|_{B^{s}_{p,r}}^2\,d\tau+a_*\biggl(1+\frac{\|Da^n\|_{L_t^\infty(B^{s-1}_{p,r})}}{a_*}\biggr)^s
\|\nabla\Pi^n\|_{L_t^1(L^p)}\biggr)\cdotp
\end{equation}
In order to ``close the estimate'', we now have to bound $\nabla\Pi^n$ in $L^p.$
For that, we apply the standard $L^p$ elliptic estimates stated in Proposition \ref{p:Lp} to \eqref{eq:Bound6}, 
and H\"older inequality so as to get
$$
\|\nabla\Pi^n\|_{L_t^1(L^p)}\leq C\biggl(\rho^*\biggl(\|f^n\|_{L^1_t(L^p)}
+\int_0^t\|u^n\|_{L^{p^*}}\|\nabla u^n\|_{L^2}\,d\tau\biggr)+\|\rho^n-\ov\rho\|_{L_t^\infty(L^{p^*})}\|\d_tu^n\|_{L_t^1(L^2)}\biggr).
$$
Note that, by Sobolev embedding, we have $$\|u^n\|_{L^{p^*}}\leq C\|u^n\|_{H^\sigma}.$$ 
So finally, there exists some  constant $C_M$ such that
$$
\|\nabla\Pi^n\|_{L^1_T(L^p)}\leq C_M.
$$
Plugging this latter inequality in \eqref{eq:Bound11}, we thus get
$$
a_*\|\nabla\Pi^n\|_{L_t^1(B^s_{p,r})}\leq C\biggl(\|f^n\|_{L_t^1(B^s_{p,r})}
+\int_0^t\|\nabla u^n\|_{B^{s-1}_{p,r}}^2\,d\tau
+a_*C_M\biggl(1+\frac{\|Da^n\|_{L_t^\infty(B^{s-1}_{p,r})}}{a_*}\biggr)^s
\biggr)\cdotp
$$
It is now easy to conclude this step: denoting
$$
U^n(t):=\|u^n(t)\|_{B^s_{p,r}}\quad\hbox{and}\quad
A^n(t):=a^*+\|Da^n(t)\|_{B^{s-1}_{p,r}},
$$
and assuming that $\ov T\leq T$ has been chosen so that 
\begin{equation}\label{eq:Bound12}
C\int_0^{\ov T}\|\nabla u^n\|_{B^{s-1}_{p,r}}\,d\tau\leq\log2,
\end{equation}
the above inequalities and \eqref{eq:Bound1}, \eqref{eq:Bound2}, 
\eqref{eq:Bound3} imply that for all $t\in[0,\ov T]$ we have
$$
A^n(t)\leq 2A_0\quad\hbox{with }\ A_0:=a^*+\|Da_0\|_{B^{s-1}_{p,r}}
$$
and
$$
U^n(t)\leq 2(U_0(t)+C\rho^*A_0\int_0^t \biggl(\|f\|_{B^s_{p,r}}+(U^n)^2
+C_Ma_*A_0^s\biggr)\,d\tau.
$$
So finally, there exists a nondecreasing function $F$ depending only on 
the norm of the data and such that for all $t\in[0,\ov T],$ we have
$$
U^n(t)\leq 2F(t)+C\rho^*A_0\int_0^t(U^n(\tau))^2\,d\tau.
$$
Therefore, if in addition
\begin{equation}\label{eq:time5}
2CA_0\int_0^{\ov T} U^n(\tau)\,d\tau\leq a_*
\end{equation}
 then we have $U^n\leq 4 F$ on $[0,\ov T].$

By arguing exactly as in the case $p\geq2,$ it is easy to see that
Condition \eqref{eq:time5}  is  satisfied if $\ov T$ is small enough
(an explicit lower bound may be obtained in terms of the data). 
So finally, we have found a positive time $T$ so that
$(a^n,u^n,\nabla\Pi^n)_{n\in\N}$ is bounded in the space $E_T.$


\subsubsection*{Step 3: Convergence of the sequence}

Let $\du^n:=u^{n+1}-u^n,$
$\dr^n:=\rho^{n+1}-\rho^n$ and 
$\dPi^n=\Pi^{n+1}-\Pi^n.$
Applying Inequality \eqref{eq:uniq} to the solutions
$(\rho^n,u^n,\nabla\Pi^n)$ and $(\rho^{n+1},u^{n+1},\nabla\Pi^{n+1})$  and
 using the uniform bounds that have been established in the
 previous step,  and \eqref{eq:ldv} ensures that there exists some $M>0$ such that for all $t\in[0,T]$ and
 $n\in\N,$ we have
\begin{equation}\label{eq:Bound8}
\|\dr^n(t)\|_{L^2}+\|\du^n(t)\|_{L^2}
\leq M\biggl(\|\dr^n(0)\|_{L^2}+\|\du^n(0)\|_{L^2}
+\int_0^t\|\df^n\|_{L^2}\,d\tau\biggr).
\end{equation}
Now, from the definition of $a_0^n$
and the mean value theorem, we get for large enough $n,$
$$
\|\dr^n(0)\|_{L^2}\leq C2^{-n}\|Da_0\|_{L^2}.
$$
Similarly, we have
$$
\|\du^n(0)\|_{L^2}+\|\df^n\|_{L^1_T(L^2)}\leq C2^{-n}
\bigl(\|\du_0\|_{L^2}+\|\df\|_{L^1_T(L^2)}\bigr).
$$
So Inequality \eqref{eq:Bound8} entails that $(\rho^n-\rho^0)_{n\in\N}$
is a Cauchy sequence in $\cC([0,T];L^2)$ and 
that $(u^n)_{n\in\N}$ is a Cauchy sequence in $\cC([0,T];L^2).$
Then, using for instance \eqref{eq:Bound6}, we see that
$(\nabla\Pi^n)_{n\in\N}$ is also a Cauchy sequence in $\cC([0,T];L^2).$

Finally, from the bounds in large norm that have been stated in the previous 
step, and the Fatou property for the Besov space,  one may conclude that
the limit $(a,u,\nabla\Pi)$ to $(a^n,u^n,\nabla\Pi^n)_{n\in\N}$
converges to some solution \eqref{eq:ddeuler} and
has the desired properties of regularity. 
As similar arguments have been used for handling the case $p\geq2,$
the details are left to the reader. 
 \smallbreak
 Let us now establish  Theorem \ref{th:BKM} in the case $p<2.$ 
 Let $(\rho,u,\nabla\Pi)$ be a solution with the properties described in
 Theorem \ref{th:main}. 
 Note that Lemma \ref{l:continuation1} is also true if $p<2.$
 So the only change lies in the proof of Lemma \ref{l:continuation2}
 which now uses the (new) lower bound for the lifespan that may be obtained
 from the computations of step 2, instead of \eqref{eq:time}. 
 This gives the first part of Theorem~\ref{th:BKM}. 
 As in the case $p\geq2,$ the last part of the proof of the theorem is a mere consequence 
 of the logarithmic interpolation inequality stated in \cite{KOT}.


\subsection{Removing the assumptions on the low frequency of the data}

As  pointed out in Section \ref{s:results}, 
in dimension $N\geq3,$ the supplementary assumption
that $(\rho_0-\ov\rho)\in L^{p^*}$ is not needed if $p>N/(N-1).$

In order to see that, one may repeat the proof of the theorem in the case $1<p\leq2.$
As before, bounding $\nabla\Pi^n$ in $L^1_T(L^p)$  is the main difficulty. 
For that, one may decompose $\nabla\Pi^n$ into two terms
$\nabla\Pi^n_1$ and $\nabla\Pi^n_2$ such that
$$
\Delta\Pi^n_1=\div(\rho^nf^n-\rho^nu^n\cdot\nabla u^n)
\quad\hbox{and}\quad
\Delta\Pi^n_2=\nabla\rho^n\cdot\d_tu^n.
$$
On the one hand, as before, one may write that
$$
\|\nabla\Pi^n_1\|_{L^p}\leq C\rho^*\bigl(\|f^n\|_{L^p}+\|u^n\|_{L^{p^*}}
\|\nabla u^n\|_{L^2}\bigr).
$$
On the other hand, we have 
$$
\nabla\Pi^n_2=(-\Delta)^{-1}\nabla(\nabla\rho^n\cdot\d_tu^n).
$$
Recall that  in dimension $N\geq2,$ the kernel of Operator $(-\Delta)^{-1}\nabla$ 
 behaves as $|x|^{1-N}.$
Hence, according to the  Hardy-Littlewood-Sobolev inequality, if $1/p+1/N<1$ then
we have 
$$
\|\nabla\Pi^n_2\|_{L^p}\leq C\|\nabla\rho^n\cdot\d_tu^n\|_{L^q}
\quad\hbox{with }\ 
\frac1q=\frac1p+\frac1N\cdotp
$$
As $(\nabla\rho^n)_{n\in\N}$ may be bounded in $\cC([0,T];L^p)$
and, by embedding, $(\d_tu^n)_{n\in\N}$ may be bounded in
$L^1([0,T];L^N),$ \emph{in terms of Sobolev norms only}, 
it is thus possible to get 
a bound of $\nabla\Pi^n_2$ in $L^1([0,T];L^p)$ in terms of 
the initial data. 
The rest of the proof goes by the steps that we used before.


\section{The proof of  Theorem \ref{th:2}}\label{s:th:2}

For $T>0,$ let us introduce the set $F_T$ of functions $(a,u,\nabla\Pi)$ such that 
$$
\begin{array}{lll}a\in\cC_b([0,T]\times\R^N),&&
Da\in\cC([0,T];B^{s-1}_{p,r}),\\[1ex]
 u\in\cC([0,T];B^{s}_{p,r}),&&
\nabla\Pi\in\cC([0,T];L^2)\cap L^1([0,T];B^{s}_{p,r}).\end{array}
$$
Uniqueness in Theorem \ref{th:2} stems from Proposition \ref{p:uniqueness}.
Indeed, we see that, as $p\leq4,$  any solution $(\rho,u,\nabla\Pi)$ in $F_T$
 satisfies $u\in\cC([0,T];W^{1,4})$ (according to Proposition \ref{p:embed} and to the remark that follows) 
  and $\nabla\Pi\in\cC([0,T];L^2).$ 
Therefore, using the velocity equation and H\"older's inequality, we get 
\begin{equation}\label{eq:1}
(\d_tu-f)=-\bigl(u\cdot\nabla u+a\nabla\Pi\bigr)\in\cC([0,T];L^2).\end{equation}
Note that, as  $u$ and $\nabla\rho$ are in $\cC([0,T];L^4),$ we have 
\begin{equation}\label{eq:2}
\d_t\rho=-u\cdot\nabla\rho\in\cC([0,T];L^2).
\end{equation} 
Now, consider two solutions $(\rho_1,u_1,\nabla\Pi_1)$
and $(\rho_2,u_2,\nabla\Pi_2)$ in $F_T,$ corresponding to the same data. 
Then \eqref{eq:1}  implies that  $\du:=u_2-u_1$ belongs to $\cC^1([0,T];L^2)$
while \eqref{eq:2} guarantees that $\dr:=\rho_2-\rho_1$ 
is in $\cC^1([0,T];L^2).$
So Proposition \ref{p:uniqueness} applies and yields uniqueness.
\medbreak
Let us now tackle the proof of the existence part of the theorem.
We claim that if we restrict our attention to solutions which are $F_T$ then 
the assumptions of  Lemma \ref{l:equiv} are fulfilled so that  it suffices to solve System \eqref{eq:ddeuler1}.
Indeed, it is only a matter of checking whether $\cQ u$ is in $\cC([0,T];L^2).$ 
Applying $\cQ$ to the velocity equation of \eqref{eq:ddeuler1}, we get 
$$
\d_t\cQ u=\cQ f-\cQ(a\nabla\Pi)-\cQ(u\cdot\nabla u).
$$
{}From the assumptions on $f,$ the definition of $F_T$ and the fact that $\cQ$ 
maps $L^2$ in $L^2,$ we see that the first two terms in the right-hand side are in $\cC([0,T];L^2).$
Concerning the last term, we just use the fact that, as pointed out above, 
 $u\cdot\nabla u$ belongs to $\cC([0,T];L^2)$ so
and $\cQ(u\cdot\nabla u),$ too.
\medbreak
Let us now go to the proof of the existence of a local-in-time solution for
\eqref{eq:ddeuler1} under the assumptions of Theorem \ref{th:2}. Compared to Theorem \ref{th:main}, 
the main change is that we do not expect to have $u\in\cC([0,T];L^2)$ any longer (i.e. 
the energy may be infinite). 
However, as the pressure satisfies 
$$
-\div(a\nabla\Pi)=\div(u\cdot\nabla\cP u)-\div f,
$$
Lemma \ref{l:laxmilgram} will ensure that $\nabla\Pi\in\cC([0,T];L^2)$ anyway
\emph{if $u\cdot\nabla\cP u$ belongs to $\cC([0,T];L^2)$}.
In view of Proposition \ref{p:embed},  Remark \ref{r:CZ} and H\"older's inequality, 
this latter property is guaranteed  by the fact that  $u\in\cC([0,T];B^{s}_{p,r})$
\emph{for some $p\leq4.$}
\smallbreak
Once this has been noticed, one may use  the same approximation 
scheme as in Theorem \ref{th:main}:
 we first set  $(a^0,u^0,\nabla\Pi^0):\equiv(a_0,u_0,0).$
Next, we assume  that $(a^n,u^n,\nabla\Pi^n)$ has been constructed over $\R^+,$
belongs to the space $F_T$ for all $T>0$ and that there exists a positive time
$T^*$ such that \eqref{eq:loinduvide} is satisfied for all $t\in[0,T^*]$ and, 
for suitable constants $C_0$ and $C$ (one can take $C_0=2C^2$),  
\begin{eqnarray}
\label{eq:bound4b}
&\|\nabla a^{n+1}(t)\|_{B^{s-1}_{p,r}}\leq  2\|\nabla a_0\|_{B^{s-1}_{p,r}}\quad\hbox{for all }\ 
t\in[0,T^*],\\\label{eq:hyp1b}
&U^n(t)\leq 4U_0(t)+C_0\rho^*A_0\Bigl(\|\div f\|_{L^1_t(B^{s-1}_{p,r})}+(\rho^*A_0)^\gamma
\|\cQ f\|_{L^1_t(L^2)}\Bigr),\\\label{eq:hyp2b}
&a_*\|\nabla\Pi^n\|_{L^1_t(B^s_{p,r})}\leq C
\biggl(\|\div f\|_{L^1_t(B^{s-1}_{p,r})}+\bigl(\rho^*A_0\bigr)^\gamma
\Int_0^t\bigl(\bigl(U^n\bigr)^2+\|\cQ f\|_{L^2}\bigr)\,d\tau\biggr)\end{eqnarray}
with   $A_0:=a^*+\|Da_0\|_{B^{s-1}_{p,r}},$
$U_0(t):=\|u_0\|_{B^s_{p,r}}+\|f\|_{L^1_t(B^s_{p,r})}$ and  
$U^n(t):=\|u^n(t)\|_{B^{s}_{p,r}}.$
\medbreak
Arguing exactly as in the proof of Theorem \ref{th:main}, we see that 
if we define $a^{n+1}$ as the solution to  $$
 \d_ta^{n+1}+u^n\cdot\nabla a^{n+1},\qquad a^{n+1}_{|t=0}=a_0
 $$
 then \eqref{eq:loinduvide} is satisfied for all time, 
  $\nabla a^{n+1}\in\cC(\R^+;B^{s-1}_{p,r})$
and 
\begin{equation}\label{eq:bound1b}
\|\nabla a^{n+1}(t)\|_{B^{s-1}_{p,r}}\leq  e^{C\int_0^tU^n(\tau)\,d\tau}\|\nabla a_0\|_{B^{s-1}_{p,r}}.
\end{equation}
So if we assume that $T^*$ has been chosen so that
\begin{equation}\label{eq:time1b}
C\int_0^{T^*}U^n(t)\,dt\leq\log2
\end{equation}
then $a^{n+1}$ satisfies \eqref{eq:bound4b}.\smallbreak
Next,  we  take $u^{n+1}$ to be  the unique solution  in $\cC(\R^+;B^s_{p,r})$ of the
 transport equation \eqref{eq:un}. 
 As before,  the right-hand side of \eqref{eq:un}   belongs to $\cC(\R^+;B^s_{p,r})$
 and one may use \eqref{eq:bound3a}.
 So finally, the existence of $u^{n+1}\in\cC(\R^+;B^s_{p,r})$ is ensured by 
Proposition \ref{p:transport}, and we have
$$\displaylines{
\|u^{n+1}(t)\|_{B^{s}_{p,r}}\leq  e^{C\int_0^tU^n(\tau)}\biggl(\|u_0\|_{B^{s}_{p,r}}
\hfill\cr\hfill+\int_0^te^{-C\int_0^\tau U^n(\tau')\,d\tau'}\bigl((\|a^{n+1}\|_{L^\infty}
+\|\nabla a^{n+1}\|_{B^{s-1}_{p,r}})\|\nabla\Pi^n\|_{B^s_{p,r}}
+\|f\|_{B^s_{p,r}}\bigr)\,d\tau\biggr).}
$$
Therefore, if we restrict our attention to those $t$ that are in $[0,T^*]$
with $T^*$ satisfying \eqref{eq:time1b}, and use Inequality \eqref{eq:bound1b}, 
we see that for all $t\in[0,T^*],$
$$
U^{n+1}(t)\leq  2U_0(t)
+CA_0
\int_0^t\|\nabla\Pi^n\|_{B^s_{p,r}}\,d\tau\quad\hbox{with }\ 
A_0:=a^*+\|\nabla a_0\|_{B^{s-1}_{p,r}}.
$$
So if we assume that $C_0=2C^2$ and that $T^*$ has been chosen so that
\begin{equation}\label{eq:time2b}
C^2\rho^*A_0\int_0^{T^*}U^n(t)\,dt\leq\frac12
\end{equation}
then taking advantage of Inequality  \eqref{eq:hyp2b}, 
we see that $u^{n+1}$ satisfies \eqref{eq:hyp1b} on $[0,T^*].$
 \smallbreak
 To finish with,  in order to  construct  the approximate pressure $\Pi^{n+1},$
 we solve the elliptic equation \eqref{eq:pres} for every positive time. 
Recall that we have $\div(u^{n+1}\cdot\nabla\cP u^{n+1})\in B^{s-1}_{p,r}$ and that
$$
\|\div(u^{n+1}\cdot\nabla\cP u^{n+1})\|_{B^{s-1}_{p,r}}\leq C (U^{n+1})^2.
$$
Next, given our assumptions on $(s,p,r)$ we have $B^s_{p,r}\hookrightarrow W^{1,4}.$
Therefore, since $\cP$ maps $L^4$ in $L^4,$ one may write
$$\begin{array}{lll}
\|u^{n+1}\cdot\nabla\cP u^{n+1}\|_{L^2}&\leq& \|u^{n+1}\|_{L^4}\|\cP u^{n+1}\|_{L^4},\\[1ex]
&\leq& C\|u^{n+1}\|_{L^4}\|u^{n+1}\|_{W^{1,4}},\\[1ex]
&\leq& C\bigl( U^{n+1}\bigr)^2.\end{array}
$$
Therefore, the second part of Proposition \ref{p:elliptic} ensures that
$\nabla\Pi^{n+1}$ is well defined in $\cC(\R^+;L^2)\cap L^1_{loc}(\R^+;B^s_{p,r})$ and that 
  $$
  a_*\|\nabla\Pi^{n\!+\!1}\|_{L_t^1(B^s_{p,r})}\leq
 C\biggl(\|\div f\|_{L^1_t(B^{s-1}_{p,r})}
 +\bigl(1+\rho^*\|Da^{n\!+\!1}\|_{L^\infty_t(B^{s-1}_{p,r})}\bigr)^\gamma
\int_0^t\bigl(U^{n+1}\bigr)^2+\|\cQ f\|_{L^2}\bigr)\,d\tau\biggr).
  $$
  Taking advantage of Inequality  \eqref{eq:bound4b} 
   at rank $n+1,$ one can now conclude that~$\nabla\Pi^{n+1}$ satisfies~\eqref{eq:hyp2b}. 
  
  At this stage we have proved that if Inequalities  \eqref{eq:bound4b}, \eqref{eq:hyp1b} and \eqref{eq:hyp2b} 
   hold for $(a^n,u^n,\nabla\Pi^n)$ then they also 
  hold for $(a^{n+1},u^{n+1},\nabla\Pi^{n+1})$
 provided $T^*$ satisfies Inequality   
  \eqref{eq:time2b}. 
  One may easily check that this is indeed the case if we set
  \begin{equation}\label{eq:timeb}
  T^*:=\sup\Bigl\{t>0\,/\, \rho^*A_0t\Bigl(U_0(t)+\rho^*A_0\|\div f\|_{L^1_t(B^{s-1}_{p,r})}
  +(\rho^*A_0)^{\gamma+1}
\|\cQ f\|_{L^1_t(L^2)}\Bigr)\leq c\Bigr\}
  \end{equation}  
  for a  small  enough constant $c$ depending only on $s,$ $p$ and $N.$
 \smallbreak
Once the bounds in $F_{T^*}$ have been established,  the last steps of the proof  are almost  
identical to those  of Theorem \ref{th:main}.
Indeed, introducing
$$
\tilde a^n(t,x):=a^n(t,x)-a_0(x)\quad\hbox{and}\quad
\tilde u^n(t,x):=u^n(t,x)-u_0(x)-\int_0^t f(\tau,x)\,d\tau
$$
and observing that $\du^n:=u^{n+1}-u^n=\tilde u^{n+1}-\tilde u^n,$
one can use exactly the same computations as before
for bounding $\da^n,$ $\du^n$ and $\nabla\dPi^n.$
As a consequence $(\tilde a^n,\tilde u^n,\nabla\Pi^n)$
is a Cauchy sequence in $\cC([0,T^*];L^2).$
Next, the bounds \eqref{eq:bound4b}, \eqref{eq:hyp1b} and \eqref{eq:hyp2b} enable us to 
show that the limit is indeed in $F_T$ and satisfies \eqref{eq:ddeuler1}. 
The details are left to the reader.
\medbreak
Let us finally establish the continuation criterion. 
Note that Lemma \ref{l:continuation1} still applies in the context of infinite energy solutions.
Hence, repeating the proof of Lemma \ref{l:continuation2} 
and using the logarithmic interpolation inequality of \cite{KOT} yields the result.
This completes the proof of Theorem~\ref{th:2}.


\section{The proof of Theorem \ref{th:3}}\label{s:th:3}

In this section, we aim at investigating the 
 well-posedness  issue of System \eqref{eq:ddeuler} in H\"older spaces $C^s$ 
 (which coincide with the Besov spaces $B^s_{\infty,\infty}$ if $s$ is not an integer), 
 and, more generally, 
 in Besov spaces of type $B^s_{\infty,r}.$  Of particular interest is the case
 of the Besov space $B^1_{\infty,1}$ which is the largest one for which Condition $(C)$ holds. 
 
 The main difficulty is that  the previous proofs where  based on the elliptic estimate
 stated in Proposition \ref{p:elliptic} which fails in the limit case $p=\infty.$ 
 In this section, we shall see that  the case of a small perturbation 
 of a constant density state may be handled by a different approach.

\subsection{The proof of uniqueness}

Note that in the case $p>4$ the solution provided by Theorem \ref{th:3} need not
satisfy $\nabla\Pi\in L^1([0,T];L^2).$ Hence $\d_tu$ need not be in $L^1([0,T];L^2)$
and the assumptions of Proposition \ref{p:uniqueness} are not satisfied. 

So, in order to prove uniqueness, we shall prove stability estimates 
in $L^p$ rather than in $L^2.$ These estimates will be also  needed 
 in the last step of the proof 
of the existence part of Theorem~\ref{th:3}. 

Consider two solutions $(\rho_1,u_1,\nabla\Pi_1)$ and $(\rho_2,u_2,\nabla\Pi_2)$ of
\eqref{eq:ddeuler}. Let $a_1:=1/\rho_1$ and $a_2:=1/\rho_2.$ As usual, denote
$\da:=a_2-a_1,$  $\du:=u_2-u_1,$ $\nabla\dPi:=\nabla\Pi_2-\nabla\Pi_1$ and $\df:=f_2-f_1.$
First, as $\div u_2=0$ and 
$$\d_t\da+u_2\cdot\nabla\da=-\du\cdot\nabla a_1,
$$ one may write
\begin{equation}\label{eq:u1}
\|\da(t)\|_{L^p}\leq\|\da(0)\|_{L^p}+\int_0^t\|\nabla a_1\|_{L^\infty}\|\du\|_{L^p}\,d\tau.
\end{equation}
Next, as 
$$
\d_t\du+u_2\cdot\nabla\du=\df-\du\cdot\nabla u_1-\da\nabla\Pi_1-a_2\nabla\dPi,
$$
we have
\begin{equation}\label{eq:u2}
\|\du(t)\|_{L^p}\leq\|\du(0)\|_{L^p}
+\int_0^t\Bigl(\|\df\|_{L^p}+\|\nabla u_1\|_{L^\infty}\|\du\|_{L^p}
+\|\nabla\Pi_1\|_{L^\infty}\|\da\|_{L^p}
+a^*\|\nabla\dPi\|_{L^p}\Bigr)\,d\tau.
\end{equation}
Finally, we notice that  $\nabla\dPi$ satisfies the elliptic equation
\begin{equation}\label{eq:u3}
\div(a_2\nabla\dPi)=\div\df-\div(\da\nabla\Pi_1)-\div(\du\cdot\nabla u_1)-\div(u_2\cdot\nabla\du).
\end{equation}
The key point here is  that, owing to $\div u_2=\div\du=0,$ we have
$$
\div(u_2\cdot\nabla\du)=\div(\du\cdot\nabla u_2).
$$
Hence Equality \eqref{eq:u3} rewrites
$$
a_*\Delta\dPi =\div\df+\div((a_*-a_2)\nabla\dPi)-\div(\da\nabla\Pi_1)-\div(\du\cdot\nabla(u_1+u_2))
$$
so that the $L^p$ elliptic estimate stated in Proposition \ref{p:Lp} implies that
$$
a_*\|\nabla\dPi\|_{L^p}\leq C\Bigl(\|\df\|_{L^p}+(a^*-a_*)\|\nabla\dPi\|_{L^p}
+\|\nabla\Pi_1\|_{L^\infty}\|\da\|_{L^p}+\|\nabla(u_1+u_2)\|_{L^\infty}\|\du\|_{L^p}\Bigr).
$$
If  the  quantity $a^*/a_*\,-1$ is small enough then  we thus have, up to a change of $C,$
$$
a_*\|\nabla\dPi\|_{L^p}\leq C\Bigl(\|\df\|_{L^p}
+\|\nabla\Pi_1\|_{L^\infty}\|\da\|_{L^p}+\|\nabla(u_1+u_2)\|_{L^\infty}\|\du\|_{L^p}\Bigr).
$$
Plugging this latter inequality in \eqref{eq:u2}
then adding up Inequality \eqref{eq:u1}, we get 
$$\displaylines{
\|(\da,\du)(t)\|_{L^p}\leq \|(\da,\du)(0)\|_{L^p}\hfill\cr\hfill
+C\int_0^t\Bigl(\|\df\|_{L^p}+\bigl(\|\nabla u_1\|_{L^\infty}+\|\nabla u_2\|_{L^\infty}+\|\nabla a_1\|_{L^\infty}+\|\nabla \Pi_1\|_{L^\infty}\bigr)\|(\da,\du)\|_{L^p}\Bigr)\,d\tau.}
$$
Applying Gronwall's lemma yields the following result which obviously implies
the uniqueness part of Theorem \ref{th:3}:
\begin{prop}\label{eq:uniq3}
Let $(\rho_1,u_1,\nabla\Pi_1)$ and $(\rho_2,u_2,\nabla\Pi_2)$ be two solutions of 
$\eqref{eq:ddeuler}$ on $[0,T]\times\R^N$ such that for some $p\in(1,\infty),$  \begin{itemize}
\item $\da:=a_2-a_1$ and $\du:=u_2-u_1$ are in $\cC([0,T];L^p),$
\item $\nabla\dPi:=\nabla\Pi_2-\nabla\Pi_1$ is in $L^1([0,T];L^p),$
\end{itemize}
and for some positive real numbers $a_*$ and $a^*$ such that $a_*\leq a^*,$
$$
a_*\leq a_1,a_2\leq a^*.
$$
There exists a constant $c$ depending only on $p$ and on $N$ such that if
$$
a^*-a_*\leq ca_*
$$
and if for all $t\in[0,T],$
$$V(t):=\int_0^t\bigl(\|\nabla u_1\|_{L^\infty}+\|\nabla u_2\|_{L^\infty}+\|\nabla a_1\|_{L^\infty}+\|\nabla \Pi_1\|_{L^\infty}\bigr)\,d\tau<\infty
$$
then  the following inequality is satisfied:
$$
\|(\da,\du)(t)\|_{L^p}\leq e^{CV(t)}\biggl(\|(\da,\du)(0)\|_{L^p}
+C\int_0^t e^{-CV(\tau)}\|\df(\tau)\|_{L^p}\,d\tau\biggr).
$$
\end{prop}


\subsection{A priori estimates}

Here we assume that $(\rho,u,\nabla\Pi)$ is a solution to \eqref{eq:ddeuler} on the time interval $[0,T]$
with the $B^s_{\infty,r}$ regularity. We want to show that if $T$ has been chosen small enough then
the size of the solution at time $t\leq T$ is of the same order as the size of the data. 

First, it is clear that we have
$$
a_*\leq a\leq a^*.$$
Moreover, one may write thanks to Proposition \ref{p:transport}:
\begin{equation}\label{eq:limit1}
\|\nabla a(t)\|_{B^{s-1}_{\infty,r}}\leq e^{C\int_0^t\|u\|_{B^{s}_{\infty,r}}\,d\tau}
\|\nabla a_0\|_{B^{s-1}_{\infty,r}},
\end{equation}
and for the velocity, we have, as in the case $p<\infty,$
\begin{eqnarray}\label{eq:limit2}&&
\|u(t)\|_{B^{s}_{\infty,r}}
\leq e^{C\int_0^t\|u\|_{B^{s}_{\infty,r}}\,d\tau}
\biggl(\|u_0\|_{B^{s}_{\infty,r}}\hspace{3cm}\nonumber
\\&&\hspace{3cm}+\int_0^t e^{-C\int_0^\tau\|u\|_{B^{s}_{\infty,r}}\,d\tau'}
\Bigl(\|f\|_{B^s_{\infty,r}}+\|a\|_{B^{s}_{\infty,r}}\|\nabla\Pi\|_{B^s_{\infty,r}}\Bigr)\,d\tau\biggr).
\end{eqnarray}
Note that applying standard $L^p$ estimates for the transport equation yields
\begin{equation}\label{eq:limit2a}
\|u(t)\|_{L^p}\leq\|u_0\|_{L^p}+\int_0^t\|f\|_{L^p}\,d\tau+ a^*\int_0^t\|\nabla\Pi\|_{L^p}\,d\tau.
\end{equation}
As Propositions \ref{p:Lp} and \ref{p:elliptic} fail in the limit case $p=\infty,$ in order to bound the pressure, 
we have to resort to other arguments. Now, dividing the velocity equation of \eqref{eq:ddeuler} by $\rho$
and applying $\div,$ we get
\begin{equation}\label{eq:limit3}
\ov a\Delta\Pi=\div f-\div(u\cdot\nabla u)+\div((\ov a-a)\nabla\Pi)\quad\hbox{with }\ \ov a:=1/\ov\rho
\end{equation}
and, by virtue of  the Bernstein inequality, we have
$$\begin{array}{lll}
\|\nabla\Pi\|_{B^s_{\infty,r}}
&\leq& \|\Delta_{-1}\nabla\Pi\|_{B^s_{\infty,r}} +\|({\rm Id}-\Delta_{-1})\nabla\Pi\|_{B^s_{\infty,r}},\\[1ex]
&\leq&C\|\nabla\Pi\|_{L^p}+\|({\rm Id}-\Delta_{-1})\nabla\Pi\|_{B^s_{\infty,r}}.
\end{array}
$$
On the one hand, in order to bound the $L^p$ norm of $\nabla\Pi,$ we simply apply the standard $L^p$ elliptic estimate (see Proposition \ref{p:Lp})
to \eqref{eq:limit3}. 
We get
$$
\ov a\|\nabla\Pi\|_{L^p}\leq C\Bigl(\|\cQ f\|_{L^p}+\|u\|_{L^p}\|\nabla u\|_{L^\infty}
+\|\ov a-a\|_{L^\infty}\|\nabla\Pi\|_{L^p}\Bigr).
$$
Hence, if $a^*/\ov a\,-1$ is small enough then 
\begin{equation}\label{eq:limit4}
\ov a\|\nabla\Pi\|_{L^p}\leq C\Bigl(\|\cQ f\|_{L^p}+\|u\|_{L^p}\|\nabla u\|_{L^\infty}\Bigr).
\end{equation}
On the other hand, for bounding the high frequency part of the pressure, one can use 
the fact that Operator $\nabla(-\Delta)^{-1}({\rm Id}-\Delta_{-1})$ is homogeneous
of degree $-1$ away from a ball centered at the origin, hence
maps $B^{s-1}_{p,r}$ in $B^s_{p,r}$ (see e.g. \cite{BCD}, Chap. 2).
Therefore we have
$$\begin{array}{lll}
\ov a\|({\rm Id}-\Delta_{-1})\nabla\Pi\|_{B^s_{\infty,r}}&\leq& C\ov a\|\Delta\Pi\|_{B^{s-1}_{\infty,r}},\\[1ex]
&\leq& C\bigl(\|\div f\|_{B^{s-1}_{\infty,r}}+\|\div(u\cdot\nabla u)\|_{B^{s-1}_{\infty,r}}+
\|\div((\ov a-a)\nabla\Pi)\|_{B^{s-1}_{\infty,r}}\bigr).\end{array}
$$
In order to bound the second term, one may combine
the Bony decomposition and the fact that $\div u=0.$
This gives
$$
\div(u\cdot\nabla u)=\sum_{i,j}\bigl(2 T_{\d_iu^j}\d_ju^i+\d_iR(u^j,\d_ju^i)\bigr).
$$
Thus applying Proposition \ref{p:op}, we may write
$$
\|\div(u\cdot\nabla u)\|_{B^{s-1}_{\infty,r}}\leq C\|\nabla u\|_{L^\infty}\|u\|_{B^s_{\infty,r}}.
$$
Finally, as $B^s_{\infty,r}$ is a Banach algebra, we have 
$$
\|\div((\ov a-a)\nabla\Pi)\|_{B^{s-1}_{\infty,r}}\leq C\|a-\ov a\|_{B^{s}_{\infty,r}}
\|\nabla\Pi\|_{B^{s}_{\infty,r}}.
$$
Putting this together with \eqref{eq:limit4}, one may conclude that
 there exists a constant $c$ such that if
\begin{equation}\label{eq:limit5}
\|a-\ov a\|_{L^\infty_T(B^{s}_{\infty,r})}\leq c\ov a
\end{equation}
then 
\begin{equation}\label{eq:limit6}
\ov a\bigl(\|\nabla\Pi\|_{L^p}+\|\nabla\Pi\|_{B^s_{\infty,r}}\bigr)
\leq C\bigl(\|\cQ f\|_{L^p}+\|\div f\|_{B^{s-1}_{p,r}}
+\|u\|_{L^p\cap B^s_{\infty,r}}\|\nabla u\|_{L^\infty}\bigr).
\end{equation}
Let us assume that  $T$ has been chosen so that
\begin{equation}\label{eq:smalltime}
C\int_0^T\|\nabla u\|_{B^{s-1}_{\infty,r}}\leq\log2
\end{equation}
and that the initial density is such that 
$$
\|a_0-\ov a\|_{B^{s-1}_{\infty,r}}\leq \frac c2\,\ov a.
$$
Then \eqref{eq:limit5} is fulfilled and, combining  Inequalities \eqref{eq:limit2}, \eqref{eq:limit2a} and
\eqref{eq:limit6}, we  get
$$
U(t)\leq 2U_0(t)+C\ov\rho\|a_0\|_{B^s_{\infty,r}}
\int_0^t\Bigl(\|\cQ f\|_{L^p}+\|\div f\|_{B^{s-1}_{p,r}}
+U^2\Bigr)\,d\tau
$$
with 
$$ U(t):=\|u(t)\|_{L^p\cap B^s_{\infty,r}}\quad\hbox{and}\quad
U_0(t):=\|u_0\|_{L^p\cap B^s_{\infty,r}}+\int_0^t\|f\|_{L^p\cap B^s_{\infty,r}}\,d\tau.
$$
It is now easy to find a time $T>0$ depending only on the data and such that 
both Condition \eqref{eq:smalltime} 
and 
$$
U(t)\leq 4U_0(t)\quad\hbox{for all }\ t\in[0,T]
$$
are satisfied.

\subsection{The proof of existence}

This is mainly a matter of making the above estimates rigorous. 
We have to be  a bit careful though    since the data which 
are considered here do not enter in the framework of Theorems \ref{th:main} and \ref{th:2}.

As a first step, we construct a sequence of smooth solutions. 
In order to enter in the Sobolev spaces framework, one may proceed as follows. 

For the density, one may consider $\rho_0^n:=\ov\rho
+S_n\bigl(\phi(n^{-1}\cdot)(\rho_0-\ov\rho)\bigr)$
where $\phi$ is a smooth compactly supported cut-off function with value $1$ 
on the unit ball of $\R^N.$ 
Obviously,  $\rho_0^n-\ov\rho$ is in $H^\infty$
 and converges weakly to $\rho_0-\ov\rho$ when $n$ goes to infinity. 
 In addition, by using the fact that $\phi$ is smooth and that $B^s_{\infty,r}$ 
 is an algebra, one may establish  that there exists some constant $C$
such that for all $n\in\N,$
$$
\|\rho_0^n-\ov\rho\|_{B^s_{\infty,r}}\leq C\|\rho_0-\ov\rho\|_{B^s_{\infty,r}}.
$$
Similarly, for the velocity, one may set  $u_0^n:=S_n (\phi(n^{-1}\cdot) u_0)$ 
and for the source term, 
$f^n:=\alpha_n\star_t \bigl(S_n(\phi(n^{-1})\cdot f)\bigr)$
where the convolution is taken with respect to the time variable only and
$(\alpha_n)_{n\in\N}$ is a sequence of mollifiers on $\R.$
\smallbreak
Applying Theorem \ref{th:main} thus provides a sequence of continuous-in-time solutions with
values in $H^\infty,$ defined on a fixed time interval. 
Then applying the above a priori estimates, it is easy to find a time $T$ independent
of $n$ for which the sequence $(\rho^n,u^n,\nabla\Pi^n)_{n\in\N}$ is bounded
in the desired space.

For proving  convergence,  one may take advantage of the stability estimates in $L^p.$
The proof is similar to that of Theorem \ref{th:main} in the case $1<p\leq2$ and is thus 
omitted. 


\subsection{A continuation criterion}

 This paragraph is dedicated to the proof of the following continuation criterion:
 \begin{prop}\label{p:continuation3}
 Assume that $s>1$ (or that $s\geq1$ if $r=1$). 
   Consider a solution $(\rho,u,\nabla\Pi)$
 to $\eqref{eq:ddeuler}$ on $[0,T[\times\R^N$ such that for some $p\in(1,\infty)$
 we have
 \begin{itemize}
 \item $\rho\in \cC([0,T);B^{s}_{\infty,r}),$
 \item $u\in\cC([0,T);B^\infty_{\infty,r}\cap L^p),$
 \item $\nabla\Pi\in L^1([0,T);B^s_{\infty,r}\cap L^p).$
 \end{itemize}
 There exists a constant $c$ depending only on $N$ and $s$ such that if for some $\ov\rho>0$ we have
 $$
\sup_{0\leq t<T}\|\rho(t)-\ov\rho\|_{B^s_{\infty,r}}\leq c\ov\rho\quad\hbox{and}\quad
\int_0^T\|\nabla u\|_{L^\infty}\,dt<\infty
$$
then $(\rho,u,\nabla\Pi)$ may be continued beyond $T$ into 
a  $B^s_{\infty,r}$ solution of $\eqref{eq:ddeuler}.$ 
 \end{prop}
 \begin{proof}
 Applying the last part of Proposition \ref{p:transport} and product estimates 
 to the velocity equation of \eqref{eq:ddeuler} yields
 for all $t\in[0,T),$
 $$\displaylines{
 \|u(t)\|_{B^s_{\infty,r}}\leq 
 e^{C\int_0^t\|\nabla u\|_{L^\infty}\,d\tau}\biggl(
 \|u_0\|_{B^s_{\infty,r}}\hfill\cr\hfill+\int_0^t e^{-C\int_0^\tau\|\nabla u\|_{L^\infty}\,d\tau'}
\Bigl(\|f\|_{B^s_{\infty,r}}+\|a\|_{B^{s}_{\infty,r}}\|\nabla\Pi\|_{B^s_{\infty,r}}\Bigr)\,d\tau\biggr).}
$$ 
Let  us bound the pressure term according to Inequality \eqref{eq:limit6}. 
Combining with \eqref{eq:limit2a} and \eqref{eq:limit4}, we eventually get
 $$\displaylines{
\|u(t)\|_{L^p\cap B^s_{\infty,r}}\leq 
 e^{C\int_0^t\|\nabla u\|_{L^\infty}\,d\tau}\biggl(
 \|u_0\|_{B^s_{\infty,r}}\hfill\cr\hfill
 +\ov\rho\int_0^t e^{-C\int_0^\tau\|\nabla u\|_{L^\infty}\,d\tau'}
 \|a\|_{B^s_{\infty,r}}
\Bigl(\|f\|_{L^p\cap B^s_{\infty,r}}+\|u\|_{L^p\cap B^s_{\infty,r}}\|\nabla u\|_{L^\infty}
\Bigr)\,d\tau\biggr).}
$$ 
So applying Gronwall's lemma ensures that $u$ belongs to 
$L^\infty([0,T);L^p\cap B^s_{\infty,r}).$
{}From this point, completing the proof is similar as for the 
previous continuation criteria.
\end{proof}
\begin{rem} 
As in the $B^s_{p,r}$ framework,  an improved continuation
criterion involving $\|\nabla u\|_{\dot B^0_{\infty,\infty}}$ instead of $\|\nabla u\|_{L^\infty}$
may be proved for the $B^s_{\infty,r}$ regularity, if  $s>1.$
The details are left to the reader.
\end{rem}


\appendix

\section{Commutator estimates}

Here we prove two estimates that have been used for estimating the pressure. 
The first result reads:
\begin{lem}\label{l:com}
Let $(s,p,r)$ satisfy Condition $(C).$
Let $\varsigma$ be in $(-1,s-1].$ 
There exists a constant $C$ depending only on $s,$ $p,$ $r,$ $\varsigma$ and $N$ such that
for all $k\in\{1,\cdots,N\},$ we have
$$
\|\d_k[a,\dq]w\|_{L^p}\leq 
Cc_q2^{-q\varsigma}\|\nabla a\|_{B^{s-1}_{p,r}}\|w\|_{B^{\varsigma}_{p,r}}\quad\hbox{for all }\ q\geq-1
$$
with $\|(c_q)_{q\geq-1}\|_{\ell^r}=1.$
\end{lem}
\begin{proof}
We follow the proof of Lemma 8.8 in \cite{D4}. 
Let $\tilde a:=a-\Delta_{-1}a.$ 
Taking advantage of the Bony decomposition \eqref{eq:bony}, we rewrite the commutator as\footnote{Recall 
the notation $T'_uv:=T_uv+R(u,v).$}
\begin{equation}\label{eq:dec}
\d_k([a,\dq]w)=
\underbrace{\d_k([T_{\tilde a},\dq]w)}_{R_q^1}+\underbrace{\d_kT'_{\dq w}\tilde a}_{R_q^2}
-\underbrace{\d_k\dq T'_w\tilde a}_{R_q^3}+\underbrace{\d_k[\Delta_{-1}a,\dq]w}_{R_q^4}.
\end{equation}
{}From the localization properties of the Littlewood-Paley decomposition, we gather that
$$
R_q^1=\sum_{|q'-q|\leq4}\d_k\bigl([S_{q'-1}\tilde a,\dq]\Delta_{q'}w\bigr).
$$
Note that $R_q^1$ is spectrally supported in an annulus of size $2^q.$ 
Hence, combining Bernstein's inequality and Lemma 2.97 in \cite{BCD}, we get
$$
\|R_q^1\|_{L^p}\leq C\sum_{|q'-q|\leq4}\|\nabla S_{q'-1}\tilde a\|_{L^\infty}\|\Delta_{q'}w\|_{L^p},
$$
whence for some sequence $(c_q)_{q\geq-1}$ in the unit sphere of $\ell^r,$
\begin{equation}\label{eq:com1}
\|R_q^1\|_{L^p}\leq C c_q2^{-q\varsigma}\|\nabla a\|_{L^\infty}\|w\|_{B^{\varsigma}_{p,r}}.
\end{equation}
To deal with $R_q^2,$ we use the fact that, owing to the  localization properties of the Littlewood-Paley decomposition, we have
$$R_q^2=\sum_{q'\geq q-2}
\d_k\bigl(S_{q'+2}\dq w\,\Delta_{q'}\tilde a\bigr).
$$ 
Hence, using the Bernstein and H\"older inequalities and the fact that
$\tilde a$ has no low frequencies, 
$$\begin{array}{lll}
\|R_q^2\|_{L^p}&\leq& C\Sum_{q'\geq q-2}\| S_{q'+2}\dq w\|_{L^\infty}
\|\Delta_{q'}\nabla\tilde a\|_{L^p},\\[1ex]
&\leq&C 2^{-q\varsigma}\,2^{q(\frac Np+1-s)}\Sum_{q'\geq q-2}2^{(q-q')(s-1)}\bigl(2^{q(\varsigma-\frac Np)}\|\dq w\|_{L^\infty}\bigr)
\bigl(2^{q'(s-1)}\|\Delta_{q'}\nabla\tilde a\|_{L^p}\bigr).
\end{array}
$$ 
Therefore, by virtue  of convolution inequalities for series and because $N/p+1-s\leq0,$ 
\begin{equation}\label{eq:com2}
\|R_q^2\|_{L^p}\leq C c_q2^{-q\varsigma}\|\nabla a\|_{B^{s-1}_{p,r}}
\|w\|_{B^{\varsigma-\frac Np}_{\infty,r}}.
\end{equation}
Next, Proposition \ref{p:op} ensures that,
under the assumptions of the lemma, the paraproduct and the remainder 
map $B^\varsigma_{p,r}\times B^s_{p,r}$ in $B^\varsigma_{p,r}.$
As moreover we have
\begin{equation}\label{eq:com3}
\|\tilde a\|_{B^s_{p,r}}\leq C\|\nabla a\|_{B^{s-1}_{p,r}},
\end{equation}
one may conclude that 
\begin{equation}\label{eq:com4}
\|R_q^3\|_{L^p}
\leq 
Cc_q2^{-q\varsigma}\|\nabla a\|_{B^{s-1}_{p,r}}\|w\|_{B^{\varsigma}_{p,r}}.
\end{equation}
Finally, as the last term $R_q^4$ is spectrally localized in a ball of size $2^q,$
Bernstein's inequality ensures that 
$$
\|R_q^4\|_{L^p}\leq C2^q \|[\Delta_{-1}a,\dq]w\|_{L^p}.
$$
Then, resorting again to Lemma 2.97 in \cite{BCD}, we get
\begin{equation}\label{eq:com5}
\|R_q^4\|_{L^p}
\leq 
Cc_q2^{-q\varsigma}\|\nabla a\|_{L^\infty}\|w\|_{B^{\varsigma}_{p,r}}.
\end{equation}
Putting Inequalities \eqref{eq:com1}, \eqref{eq:com2}, \eqref{eq:com4} and
 \eqref{eq:com5} together and using Proposition \ref{p:embed}
 completes the proof of the lemma.
\end{proof}
\begin{lem}\label{l:combis}
Let $\varsigma>0$ and $1\leq p,r\leq\infty.$ There exists a constant $C$ such that 
$$
\|\d_k[a,\dq]w\|_{L^p}\leq 
Cc_q2^{-q\varsigma}\Bigl(\|\nabla a\|_{L^\infty}\|w\|_{B^{\varsigma}_{p,r}}
+\|w\|_{L^\infty}\|\nabla a\|_{B^\varsigma_{p,r}}\Bigr)
\quad\hbox{for all }\ q\geq-1
$$
with $\|(c_q)_{q\geq-1}\|_{\ell^r}=1.$
\end{lem}
\begin{proof}
We use again Decomposition \eqref{eq:dec}.
We have already proved in \eqref{eq:com1} and \eqref{eq:com5}  that $R_q^1$ and $R_q^4$ satisfy the desired inequality.
Concerning $R_q^2,$  recall that
$$
\|R_q^2\|_{L^p}\leq C\sum_{q'\geq q-2} \|S_{q'-1}\dq w\|_{L^\infty}\|\Delta_{q'}\nabla a\|_{L^p},
$$
 whence
 $$
\|R_q^2\|_{L^p}\leq C2^{-q\varsigma}
\sum_{q'\geq q-2} 2^{(q-q')\varsigma}\,\|w\|_{L^\infty}\,2^{q'\varsigma}\|\Delta_{q'}\nabla\tilde a\|_{L^p}.
$$
As $\varsigma>0,$ convolution inequalities for series yield the desired inequality for $R_q^2.$
\smallbreak
According to Proposition \ref{p:op}, we have 
$$
\|T'_{w}\tilde a\|_{B^{\varsigma+1}_{p,r}}\leq C\|w\|_{L^\infty}\|\tilde a\|_{B^{\varsigma+1}_{p,r}}.
$$
Hence, as $\|\tilde a\|_{B^{\varsigma+1}_{p,r}}\leq C\|\nabla a\|_{B^{\varsigma}_{p,r}},$
the term  $R_q^3$ satisfies the required inequality. 
\end{proof}


\section{A Bernstein-type inequality}

\begin{lem}\label{l:Bernstein}
Let $1<p<\infty$  and $u\in L^p$ such that ${\rm Supp}\,\widehat u
\subset \{\xi\in\R^N\,/\, R_1\leq|\xi|\leq R_2\}$ for some  real numbers
$R_1$ and $R_2$ such that $0<R_1<R_2.$ Let $a$ be a bounded measurable function over
$\R^N$ such that $a\geq a_*>0$ a.e. 
There exists a constant
$c$ depending only on $N$ and $R_2/R_1$, and such that
\begin{equation}\label{eq:bernstein} ca_*\biggl(\frac{p-1}{p^2}\biggr)R_1^2\int|u|^p\,dx\leq(p-1)\int a|\nabla u|^2|u|^{p-2}\,dx
=-\int \div(a\nabla u)\,|u|^{p-2}u\,dx.
\end{equation}
\end{lem}
\begin{proof}
The case $a\equiv1$ has been treated in   
\cite{D2} and  readily entails  the left inequality in \eqref{eq:bernstein}
for one may write, owing to the case $a\equiv1,$  
$$
c\,a_*\,R_1^2\int|u|^p\,dx\leq p^2\int a_*|\nabla u|^2|u|^{p-2}\,dx
\leq  p^2\int a|\nabla u|^2|u|^{p-2}\,dx.
$$
Let us now justify  the right equality in \eqref{eq:bernstein}. 
In the case $p\geq2,$ it stems from a straightforward integration by parts. 

Let us focus on the case $1<p<2$ which is more involved. Smoothing out $a$ if
needed, one may assume with no loss of generality that $a$ is  in $W^{1,\infty}.$
Let  $T_\eps(x)=\sqrt{x^2+\eps^2}$ for $x\in\R$ and $\eps>0$. We have 
$$
\displaylines{
-\int_{\R^N}\div(a\nabla u) (T_\eps(u))^{p-1}T'_\eps(u)\,dx=(p-1)\int_{\R^N}
a|\nabla u|^2|T'_\eps(u)|^2\bigl(T_\eps(u)\bigr)^{p-2}\,dx\hfill\cr\hfill
 +\int_{\R^N}
a|\nabla u|^2T^{''}_\eps(u)(T_\eps(u))^{p-1}\,dx.\quad}
$$
In view of the monotonous convergence theorem, 
$$
\lim_{\eps\rightarrow0}\int_{\R^N}
a|\nabla u|^2|T'_\eps(u)|^2\bigl(T_\eps(u)\bigr)^{p-2}\,dx
=\int_{\R^N}
a|\nabla u|^2\,|u|^{p-2}\,dx\in \overline \R^+.
$$
Next, we notice that 
 \begin{equation}\label{eq:L1}
 |\div(a\nabla u) (T_\eps(u))^{p-1}T'_\eps(u)|\leq
|u|^{p-1}|\div(a\nabla u)|.\end{equation}
  Now, as $a\in W^{1,\infty}$  and $u$ is a smooth function 
with all derivatives in $L^p$
(owing to the spectral localization), 
one may write
$$
\div(a\nabla u)=a\Delta u+\nabla a\cdot\nabla u,
$$
hence $\div(a\nabla u)$ is in $L^p$ 
and the right-hand side of \eqref{eq:L1} is an integrable function. 
So finally   Lebesgue's  dominated
convergence theorem entails that
$$
\lim_{\eps\rightarrow0}\int_{\R^N}\div(a\nabla u) (T_\eps(u))^{p-1}T'_\eps(u)\,dx=
\int_{\R^N}u|u|^{p-2}\div(a\nabla u)\, dx.
$$
Therefore 
\begin{equation}\label{eq:A1}
(p-1)\int_{\R^N}
a|\nabla u|^2\,|u|^{p-2}\,dx\leq -\int_{\R^N}u|u|^{p-2}\div(a\nabla u)\, dx<\infty.
\end{equation}
In fact, equality does hold.
Indeed, whenever $x\not=0$, the term $T^{''}_\eps(x)T_\eps(x)^{p-1}$ tends to $0$
when $\eps$ goes to $0$ and $$
T^{''}_\eps(x)T_\eps(x)^{p-1}=|x|^{p-2}\Frac{(\eps/x)^2}{\bigl(1+(\eps/x)^2\bigr)^{2-{\frac p2}}}\leq |x|^{p-2}.
$$
Therefore, 
as, according to \eqref{eq:A1}, the function   $|\nabla u|^2\,|u|^{p-2}$ is integrable over $\R^N,$
we get $$
 \lim_{\eps\rightarrow0}\int_{u\not=0}a|\nabla
u|^2T^{''}_\eps(u)(T_\eps(u))^{p-1}\,dx=0. $$
On the other hand, as $u$ is real analytic,
$$
\int_{u=0}a|\nabla
u|^2T^{''}_\eps(u)(T_\eps(u))^{p-1}\,dx=\eps^{p-2}\int_{u=0}a|\nabla
u|^2=0.$$
\end{proof}


\begin{thebibliography}{xxx}
\bibitem{AP} H. Abidi and M. Paicu: Existence globale pour un fluide inhomog\`ene,
{\it   Ann. Inst. Fourier,} {\bf 57}(3),  883--917 (2007).
\bibitem{BCD} H. Bahouri, J.-Y. Chemin and R. Danchin: 
{\it Fourier Analysis and Nonlinear Partial Differential Equations,} 
Springer, to appear.
\bibitem{BKM} J. Beale, T.  Kato and A.  Majda:
Remarks on the breakdown of smooth solutions for the $3$-D Euler equations,
{\it Communications in  Mathematical  Physics}, {\bf 94}(1),  61--66 (1984). 
\bibitem{BV1}
H. Beir\~ao da Veiga and A. Valli: On the Euler equations for nonhomogeneous fluids. I. 
{\it  Rend. Sem. Mat. Univ. Padova}, {\bf 63}   151--168 (1980). 
\bibitem{BV2}
H. Beir\~ao da Veiga and A. Valli:
 On the Euler equations for nonhomogeneous fluids. II. {\it J. Math. Anal. Appl.} {\bf 73}(2), 338--350 (1980). 
  \bibitem{BV3}
H. Beir\~ao da Veiga and A. Valli: Existence of $C^{\infty }$ solutions of the Euler equations for nonhomogeneous fluids. {\it Communications in Partial Differential Equations}, 
{\bf 5}(2), 95--107 (1980).
 
\bibitem{Bony} J.-M. Bony: 
{\it Calcul symbolique et propagation des singularit\'es pour
les \'equations aux d\'eriv\'ees partielles non lin\'eaires,}
Ann. Sci. \'Ecole Norm. Sup. {\bf 14}(4), 209--246 (1981).

\bibitem{D1}
R. Danchin: Global existence in critical spaces for compressible
Navier-Stokes equations, {\it Inventiones Mathematicae},
{\bf 141}(3),  579--614 (2000).
\bibitem{D2} R. Danchin: Local theory in critical spaces for 
compressible viscous and heat-conductive gases, {\it Communications in
Partial  Differential Equations}, {\bf 26},   1183--1233 (2001)
and Erratum: {\bf  27},  2531--2532 (2002). 
\bibitem{D3} 
 R. Danchin: A few remarks on the Camassa-Holm
equation, {\it Differential and Integral Equations},
{\bf 14}, pages 953--988 (2001).
\bibitem{D4} 
R. Danchin: The inviscid limit for 
density dependent incompressible fluids, 
 {\it Annales de la Facult\'e des Sciences de Toulouse}, {\bf 15}, 
 pages 637--688 (2006).
\bibitem{grafakos} L. Grafakos:
{\em Classical and Modern Fourier Analysis}, Prentice Hall, 2006.



\bibitem{Itoh} S. Itoh: Cauchy problem for the Euler equations of a nonhomogeneous ideal incompressible fluid. II. {\it Journal of the  Korean Mathematical  Society,}
 {\bf 32}(1)  1, 41--50 (1995). 
 
 \bibitem{IT} S. Itoh and A. Tani: Solvability of nonstationary problems for nonhomogeneous incompressible fluids and the convergence with vanishing viscosity, 
  {\it Tokyo Journal of  Mathematics,} {\bf  22}(1), 17--42 (1999).
  
  \bibitem{KOT} H. Kozono,  T. Ogawa and Y. Taniuchi: The critical Sobolev inequalities
  in Besov spaces and regularity criterion to some semi-linear evolution equations,
  {\it Mathematische Zeitschrift}, {\bf 242}, 251--278 (2002).
  
  \bibitem{M} J. Marsden:
Well-posedness of the equations of a non-homogeneous perfect fluid.
{\it Comm. Partial Differential Equations,} {\bf 1}(3), 215--230 (1976). 

\bibitem{Meyers} N. Meyers:  An $L^{p}$-estimate for the gradient of solutions of second order elliptic divergence equations, {\it  Annali della  Scuola Normale  Superiore di Pisa}, {\bf 17},  189--206 (1963).

\bibitem{RS} T. Runst and W. Sickel:
{\em Sobolev spaces of fractional order, Nemytskij operators, and nonlinear
partial differential equations},
Nonlinear Analysis and Applications, 3. Walter de Gruyter \& Co., Berlin, 1996.

  \bibitem{VZ} A.  Valli and W. Zaj\c aczkowski: About the motion of nonhomogeneous ideal incompressible fluids.  {\it Nonlinear Analysis, TMA,} {\bf  12}(1), 43--50 (1988).
  \bibitem{V} M. Vishik: Hydrodynamics in Besov spaces,
{\em Archive for Rational Mechanics and  Analysis}, {\bf 145},
 197--214 (1998).
 
\bibitem{Z} Y. Zhou:
Local well-posedness for the incompressible Euler equations in the critical Besov spaces.
{\it Annales de l'Institut  Fourier,} {\bf 54}(3), 773--786 (2004). 

\bibitem{Zper} Y. Zhou: Personnal communication (2005). 
\end{thebibliography}
\end{document}